\documentclass[11pt]{amsart}
\usepackage{amsmath,amssymb,amsthm,mathrsfs}
\usepackage{epsfig,color}
\usepackage[latin1]{inputenc}
\usepackage[english]{babel}
\usepackage{mathtools}
\usepackage{braket}
\usepackage[toc]{appendix}
\usepackage{enumitem}
\usepackage{ulem}
\usepackage{graphicx}
\usepackage{xfrac, nicefrac}
\usepackage{csquotes}
\usepackage{comment}
\usepackage{lowco2}
\usepackage{esint}
\usepackage{hyperref}
\allowdisplaybreaks

\numberwithin{equation}{section}

\DeclarePairedDelimiter{\norm}{\lVert}{\rVert}

\newtheorem*{theorem*}{Theorem}
\newtheorem{theorem}{Theorem}[section]
\newtheorem{lemma}[theorem]{Lemma}
\newtheorem{proposition}[theorem]{Proposition}
\newtheorem*{proposition*}{Proposition}
\newtheorem{example}[theorem]{Example}
\newtheorem{corollary}[theorem]{Corollary}
\newtheorem{remark}[theorem]{Remark}
\newtheorem*{remark*}{Remark}

\newtheorem{definition}[theorem]{Definition}

\usepackage{graphicx}
\usepackage{tikz}
\usepackage{pgfplots}

\newcommand{\field}[1]{\mathbb{#1}}
\providecommand{\N}{\field{N}}

\providecommand{\R}{\field{R}}

\usepackage{pgf,tikz,pgfplots}
\usepackage{mathrsfs}
\usetikzlibrary{arrows}
\usepackage{float}

\pgfplotsset{compat=1.18} 
\title{A sharp monomial Caffarelli-Kohn-Nirenberg inequality}
\begin{document}
	
	\begin{abstract}
		We consider a monomial Caffarelli-Kohn-Nirenberg inequality, find the optimal constant and classify the optimizers under an integrated curvature dimension condition. We take advantage of the $\Gamma$-calculus to exploit geometrical techniques to tackle the problem and regularity results to justify some integration by parts. A symmetry-breaking result is also provided. 
	\end{abstract}
	
	\author{Francesco Pagliarin \textsuperscript{1}}
	\address{\textsuperscript{1} Institut Camille Jordan, Universit\'e Claude Bernard Lyon 1, 43 boulevard du 11 Novembre 1918
		69622 Villeurbanne cedex, France}
	\email{pagliarin@math.univ-lyon1.fr}

	
	\maketitle
	\section{Introduction}
	In their paper \cite{Ca}, Cabr\'e and Ros-Oton proved a Sobolev inequality with monomial weight $x^A$, finding also the optimal constant and some optimizers, of the form
	\begin{equation}\label{eqcabros}
		\left(\int_{\R^d_*} |f|^{2^*} x^A dx\right)^{2/2^*}\leq C\int_{\R^d_*} |\nabla f|^2 x^A dx\qquad\forall f\in C^\infty_c(\R^d)
	\end{equation}
	where 
	\begin{align}\label{eq1.2}
		\begin{split}
			x^A:=\prod_{i=1}^d|x_i|^{A_i} \quad&\text{is the monomial weight, with }A_i\in \R_+\cup \{0\}\\
			\R^d_*:=\{x\in \R^d: x_i>0\text{ if }A_i >0\}\quad&\text{is the domain of integration}\\
			D:=d+\sum_{i=1}^{d}A_i=:d+|A|\quad&\text{is the monomial dimension}
		\end{split}
	\end{align}
	and
	\[
	2^*:=\frac{2D}{D-2}\quad\text{is the Sobolev's exponent}.
	\]
	This inequality arises while studying some regularity results of elliptic equations defined on domains having certain symmetries. More precisely, in \cite{CabReg}, the same authors employ \eqref{eqcabros} with $d=2$ to prove $L^p$ estimates for stable solutions to some semilinear elliptic equations. 
	The authors find the optimal constant in \eqref{eqcabros} at first proving a monomial isoperimetric inequality, adapting the ABP method as done in \cite{Cab} by Cabr\'e. After this, they use a classical symmetrization technique exploited also by Talenti in \cite{Talenti} for the standard Sobolev inequality. They indeed prove that some optimizers are radially symmetric and have a determined structure. The optimal constant is then easy to compute explicitly being it the quotient of two 1 dimensional integrals. The surprising thing is that, even though the monomial weight is not radially symmetric, all the optimizers have this symmetry up to translations. Observe that, if $u(x):=f(|x|)$ is an optimizer and $A_j=0$ for some $j$, also $f(|x-e_j|)$ is an optimizer, where $e_j$ si the $j$-th vector of the canonical base in $\R^d$.\\
	Another remarkable inequality is Caffarelli-Kohn-Nirenberg's, proved in \cite{Caff}
	\begin{equation}\label{eqckn}
		\left(\int_{\R^d} \frac{|f|^q}{|x|^{bq}}dx\right)^{2/q}\leq C_{a,b}\int_{\R^d}\frac{|\nabla f|^2}{|x|^{2a}}dx\qquad\forall f\in C^\infty_c(\R^d\setminus \{0\})
	\end{equation}
	which is valid if 
	\[
	0\leq b-a\leq 1,\quad a\neq \frac{d-2}{2}.
	\]
	The peculiarity of this inequality, as opposed to \eqref{eqcabros}, is that under some hypotheses on the parameters $a,b$, optimizers cannot be radially symmetric, even though the weights have this symmetry. This phenomenon, called ``symmetry breaking'', was completely described in \cite{Dolbeault} by Dolbeault, Esteban and Loss after having been an open problem for many years. In particular, calling $\alpha:=1+a-\frac{bq}{2}$ and $n=\frac{2q}{q-2}$, the authors established that symmetry breaking occurs if and only if $\alpha^2>\frac{d-1}{n-1}$. Recall that $\alpha^2\leq \frac{d-1}{n-1}$ is also known as ``Felli-Schneider'' condition, from the work \cite{FelliSc} of Felli and Schneider. \\ \ \\
	These different phenomena and opposed behaviour of optimizers of the two inequalities brought us studying a more general one that embodies \eqref{eqcabros} and \eqref{eqckn} as limiting cases. Indeed, in this paper we study the following monomial Caffarelli-Kohn-Nirenberg inequality
	\begin{equation}\label{eqmonckn}
		\left(\int_{\R^d_*}\frac{|f|^p}{|x|^{bp}}x^A dx\right)^{2/p}\leq C\int_{\R^d_*} \frac{|\nabla f|^2}{|x|^{2a}}x^Adx\qquad\forall f\in C^\infty_c(\R^d\setminus \{0\})
	\end{equation}
	finding the optimal constant and classifying optimizers under a curvature bound generalizing the Felli-Schneider condition.
	The approach we follow is based on the $\Gamma$-calculus and it is similar to the one exploited in \cite{Zug} by Dupaigne, Gentil, Zugmeyer. 
	However, unlike the classical CKN inequality, the monomial component adds a difficulty. As a matter of fact, at some point of the paper we deal with solutions to degenerate elliptic equations of the form 
	\begin{equation*}
		-{\rm div}(x^A \nabla u)=x^A f
	\end{equation*}
	for which classical elliptic regularity does not apply. Fortunately, recent results of Terracini, Sire, Vita \cite{ SirTerVit21a, SirTerVit21b, TTV} (for the case $A_1\neq 0$, $A_j=0$ for all $j\geq 2$) and Cora, Fioravanti, Vita and the author \cite{CoraFiorPagliaVita} prove Schauder's estimates for these degenerate/singular elliptic equations. We will take advantage of it to justify some integrations by parts later on.\\
	In addition, the method we follow is slightly the opposite compared to the aforementioned classical one for the study of optimal constants and optimizers in functional inequalities. Indeed, rather than first determining some optimizers and then computing the optimal constant, we begin by providing the optimal constant and only after we find necessary conditions satisfied by the optimizers. This method, which does not employ any symmetrization technique, could allow us to classify all the optimizers.\\ \ \\
	Let us briefly review some recent results on monomial functional inequalities. In \cite{Castro}, Castro proved the validity of a monomial Sobolev inequality where the monomial weight in the two integrals is allowed to have different powers. Moreover, the domain of integration is not $\R^d_*$, but it can be an arbitrary hyperoctagon in the euclidean space. The trade-off is that it does not seem evident computing the optimal constant with this method. \\
	Apart from Sobolev's, other inequalities with monomial weights have been studied recently. In \cite{Ca} for instance, the authors deal also with a Trudinger-Moser inequality with monomial weights, whose sharp version was proved later in \cite{Lam} by Lam. Moreover, one of the most significant inequalities is the isoperimetric one. In \cite{CaRosSer} Cabr\'e, Ros-Oton, Serra extend the monomial isoperimetric inequality in \cite{Ca} to the case of more general homogeneous weights. Eventually, in \cite{Cinti} Cinti, Glaudo, Pratelli, Ros-Oton, Serra manage to characterize all the optimizers of isoperimetric inequalities defined on cones with  weights $w$ that are $\beta$ homogeneous and such that $w^\frac{1}{\beta}$ is concave. In this way they classify all the optimizers also for the monomial isoperimetric inequality, which was an open question from \cite{Ca}. It is also observed in \cite{CaRosSer} that the hypothesis of the concavity of $w^{\frac{1}{\beta}}$ is equivalent to a pointwise curvature dimension condition. We will see later that, at least for Sobolev's inequality, this hypothesis is not optimal, since we will work under a weaker hypothesis of integrated curvarture dimension condition.
	\\ \ \\
	Before delving into our results, it is necessary to introduce some notation. As in \cite{Ca}, we keep the same definition \eqref{eq1.2} for the domain of integration $\R^d_*$ and the monomial weight $x^A$ observing that we will always consider the exponents $A_i\geq 0$. Recall that the monomial dimension $D$ is defined as the sum between the topological dimension $d$ and the exponents in $x^A$, namely
	\[
	D:=d+\sum_{i=1}^d A_i=:d+|A|.
	\]
	We prove in the Appendix \ref{sec4} that \eqref{eqmonckn} is valid for the following range of parameters
	\begin{align}\label{eq1.3}
		0\leq b-a\leq 1\quad \text{and }a\neq a_c:=\frac{D-2}{2}.
	\end{align}
	However, for simplicity we do not consider the case $b=a+1$ (monomial Hardy inequality, see  \cite{Duy} by Duy, Lam, Lu) and we stay in the range $a<a_c$ since the case $a>a_c$ can be treated adapting the property of modified inversion symmetry of the classical CKN inequality,  see \cite{CatrinaWang} by Catrina and Wang. Thus, from now on we consider $a,b$ such that 
	\begin{equation}\label{eqhpab}
		0\leq b-a<1,\quad a<a_c
	\end{equation}
	In addition, by a classical scaling argument, the Sobolev exponent $p$ in \eqref{eqmonckn} is fixed and equal to
	\begin{align*}
		p=\frac{2D}{D-2+2(b-a)}.
	\end{align*}
	As in the standard Sobolev's inequality, there exists a number $n\in \R_+$, not necessarily an integer, such that
	\begin{equation*}
		p=\frac{2n}{n-2}.
	\end{equation*}
	It is then easy to verify that the three notions of dimensions that we just introduced are ordered as follows
	\begin{equation*}
		d\leq D\leq n.
	\end{equation*}
	Lastly, define another parameter, $\alpha$ as
	\begin{equation*}
		\alpha:=1+a-\frac{bp}{2}
	\end{equation*}
	which, under \eqref{eqhpab}, is such that $\alpha > 0$. \\
	Observe that from the definitions it is easy to verify that the following equalities hold
	\begin{align}\label{eq1.12}
		\begin{cases}
			\alpha n=D-bp \\
			2a=(D-2)-\alpha(n-2).\\
		\end{cases}
	\end{align}
Define $K:=\{i\in\{1,...,d\} \text{ such that } A_i\neq 0\}$ and $k:=|K|$.
	We are now in the position to state our main result.
	\begin{theorem}\label{thm1.1}
		Under hypotheses \eqref{eqhpab}, $n>4$, $A_d=0$ and
		\begin{equation}\label{FScond}
			\alpha^2\leq \frac{D-1}{n-1}
		\end{equation}
		the optimal constant $C_{opt}$ in \eqref{eqmonckn} is equal to 
		\begin{equation}\label{eq1.14}
			C_{opt}=\frac{4}{\alpha^2 n(n-2)Z^{\frac{2}{n}}}
		\end{equation}
		where $$Z=\frac{ D\sqrt{\pi}}{\alpha}\frac{\Gamma(\frac{A_1+1}{2})\cdot...\cdot \Gamma(\frac{A_d+1}{2})}{2^k \Gamma(1+\frac{D}{2})} \frac{\Gamma\left(\frac{n}{2}\right)}{\Gamma\left(\frac{n+1}{2}\right)}.$$  Moreover, if $\alpha^2\leq  \frac{D-1}{n-1}<1$, then all the optimizers are radial and they have the following form
		\begin{equation}\label{eqopt}
			f(x)=\left(\frac{1}{s+t|x|^{2\alpha}}\right)^\frac{n-2}{2}\quad\text{with }s,t>0.
		\end{equation}
	\end{theorem}
	Let us briefly comment on the hypotheses of the theorem. At first, we want $n>4$ to prove some regularity results 
	and density theorems of smooth functions with compact support in a suitable function space. However, it might not be an optimal hypothesis. Secondly, \eqref{FScond} has a geometrical meaning since it provides an integrated bound on the curvature of our weighted manifold, see \eqref{eqintcd} below. In the classical CKN inequality, as specified in \cite{Dolbeault}, with $d=D$, \eqref{FScond} reduces to the aforementioned Felli-Schneider condition. Moreover, in \cite{DSZ} Dupaigne, Gentil, Simonov, Zugmeyer proved that if $\alpha^2>\frac{D-1}{n-1}$ optimizers cannot be radial. Thus, hypothesis \eqref{FScond} is optimal. Finally, $A_d=0$ allows us to describe our space as a warped product of an interval with a weighted sphere and take advantage of the geometrical results proved in \cite{DSZ}, see \eqref{eqwarped} below. The case where all the exponents $A_j$ are diffenent from zero is still unknown and might be treated in a following paper.
	Eventually, let us insist that \eqref{FScond} provides an integrated curvature dimension condition, which is optimal for \eqref{eq1.14} to hold and it is weaker than the concavity of $w^{\frac{1}{\beta}}$ in \cite{CaRosSer}, as commented above.
	\begin{remark}
		The classification of optimizers in the case $\alpha^2=\frac{D-1}{n-1}=1$, where the inequality reduces to the monomial Sobolev one \eqref{eqcabros}, is more delicate. Indeed, if we wanted to apply the same approach, the classification of optimizers would be linked to the classification of eigenfunctions for the eigenvalue $D-1$ of the operator $-L_\theta$, defined later. Unfortunately, we are not able for the moment to provide it. Instead, this case has already been treated by Nguyen in \cite{Ngu} using other techniques. The author adapts the ones belonging to optimal transport developped in \cite{CorderoNazVillani} by Cordero, Nazar\'e, Villani to prove the sharp version the of monomial Sobolev inequality and to classify the optimizers. The proof is essentially based on the existence of the Brenier map where no arguments of compactness or symmetrization are employed. Eventually, in \cite{Ngu} the author proved that all the optimizers in \eqref{eqcabros} are equal to functions of the form \eqref{eqopt} up to a traslation with respect to the variables which are not charged by the monomial weight.
	\end{remark}
	To conclude, we prove a corollary for the CKN inequality on the whole space $\R^d$ where a symmetry breaking phenomenon occurs. Consider
	\begin{equation}\label{eqsobRd}
		\left(\int_{\R^d}\frac{|f|^p}{|x|^{bp}}x^A dx\right)^{2/p}\leq C\int_{\R^d} \frac{|\nabla f|^2}{|x|^{2a}}x^Adx\qquad\forall f\in C^\infty_c(\R^d\setminus \{0\}). 
	\end{equation}
	 Up to relabelling, suppose that only the first $k$ coordinates are charged by non-vanishing exponents $A_i$ and define the $2^k$ Weyl chambers $\R^d_\epsilon:=\bigcap_{i\in K}\{\epsilon_i x_i>0\}$ where $\epsilon\in \{-1,1\}^k$.
	\begin{corollary}
		Under hypotheses \eqref{eqhpab}, $n>4$, $A_d=0$ and
		\begin{equation*}
			\alpha^2\leq \frac{D-1}{n-1},
		\end{equation*}
		\eqref{eqsobRd} holds with 
		\begin{equation}\label{eqoptC}
			C=\frac{4}{\alpha^2 n(n-2)Z^{\frac{2}{n}}}=:C_{opt, \R^d_*},
		\end{equation}
	the optimal constant of Theorem \ref{thm1.1}.
		In addition, if
		\begin{equation}\label{hpAi}
			A_i\notin (0,1)\quad  \forall  i,
		\end{equation}
		then \eqref{eqoptC} is optimal and all the optimizers are equal to
		\begin{align}\label{eqoptRd}
			f(x)=\left(\frac{1}{s+t|x|^{2\alpha}}\right)^\frac{n-2}{2}\chi_{\R^d_\epsilon}\qquad\text{with }s,t>0,\quad\epsilon\in \{-1,1\}^{k}
		\end{align}
		where $\chi_{\R^d_\epsilon}$ denotes the characteristic function on the Weyl chamber $\R^d_\epsilon$.
	\end{corollary}
	We do not know whether \eqref{eqoptC} is still optimal if we remove hypothesis \eqref{hpAi}. If it were, however, it could not be attained. As a matter of fact, functions of the form \eqref{eqoptRd}, the only candidates to be optimizers, do not belong to the suitable weighted Sobolev space on $\R^d$.\\
	In any case, this result is rather surprising since it tells that, under suitable hypotheses, no optimizer for \eqref{eqsobRd} can be radial on the whole space $\R^d$. The proof follows.
	\begin{proof}
		Consider a function $\varphi\in C^\infty_c(\overline{\R^d_*}\setminus\{0\})$. 
		Since we work under hypothesis \eqref{hpAi}, \cite[Lemma 3.3 and Lemma 3.4]{CoraFiorPagliaVita} tell that $\varphi$ can be approximated in the suitable weighted Sobolev space by functions $\varphi_h\in C^\infty_c(\R^d_*)\subset C^\infty_c(\R^d\setminus\{0\})$. Therefore,
		\begin{align*}
			C_{opt,\R^d_*}\leq C_{opt,\R^d}
		\end{align*}
		where we denote by $C_{opt,\R^d}$ the optimal constant of \eqref{eqsobRd}.
		So, if we prove that \eqref{eqsobRd} is valid with $C=C_{opt,\R^d_*}$, this is necessarily optimal. Let us now reason as in \cite[Proposition 6.6]{Veli}. For $u\in  C^\infty_c(\R^d\setminus\{0\})$ we have
		\begin{align}\label{eq1.16}
			\begin{split}
				&\left(\int_{\R^d} \frac{|u|^p}{|x|^{bp}} x^Adx\right)^\frac{2}{p}=\left(\sum_{\epsilon\in\{-1,1\}^{k}} \int_{\R^d_\epsilon}\frac{|u|^p}{|x|^{bp}} x^Adx\right)^\frac{2}{p}\\&\leq C_{opt,\R^d_*}\left(\sum_{\epsilon\in\{-1,1\}^{k}} \left(\int_{\R^d_\epsilon} \frac{|\nabla u|^2}{|x|^{2a}}x^A dx\right)^\frac{p}{2}\right)^\frac{2}{p}\leq C_{opt,\R^d_*} \int _{\R^d} \frac{|\nabla u|^2}{|x|^{2a}}x^A dx
			\end{split}
		\end{align}
		where in the last inequality we used that $\norm{y_n}_{l^\frac{p}{2}(\N)}\leq \norm{y_n}_{l^1(\N)}$ for any sequence $\{y_n\}_{n\in\N}\in l^1(\N)$. This implies in turn that $C_{opt,\R^d}=C_{opt,\R^d_*}$.  Let us suppose now that in \eqref{eq1.16} we have equality everywhere.  Since $\norm{y_n}_{l^\frac{p}{2}(\N)}= \norm{y_n}_{l^1(\N)}$ if and only if there exist $j$ such that $y_i=0$ for all $i\neq j$, $u$ needs to be supported on one Weyl chamber only. Eventually, Theorem \ref{thm1.1} allows us to conclude that all the optimizers of \eqref{eqsobRd} are of the form \eqref{eqoptRd}.
	\end{proof}
	The paper is organized as follows. 
	In section \ref{sec2} we introduce the notion of $\Gamma$-calculus and the geometric setting of weighted manifolds where we set the problem. In addition, the integrated curvature dimension condition is proved.  In section \ref{sec3} we prove our main theorem and eventually in the Appendix \ref{sec4} we prove some technical lemmas and the validity of \eqref{eqmonckn} with a constant, which need not be optimal.
	
	\section{Geometric, functional setting and curvature}\label{sec2}
	The first part of this section is devoted to introducing the $\Gamma$-calculus notation. We begin by the definition of weighted manifold and associated operators.
	\begin{definition}
		A weighted manifold is the given of a triple $(M,g,d\mu)$ where $(M,g)$ is a smooth Riemannian manifold 
		and $d\mu=e^{-W}dV_g$ \footnote{Recall that $dV_g$ is the volume form induced by the metric $g$. In a fixed chart it is $dV_g=\sqrt{|g|}dx$ where $dx$ is the Lebesgue measure and $|g|=\text{det}(g_{ij})$. By $g^{ij}$ we denote the inverse of the metric.}is a measure associated to it where $W$ is a smooth function in $M$.
		Fix now $f,h$ smooth functions on $M$ and define the following operators
		\begin{align*}
			\text{Carr\'e du champ:}&\quad\Gamma(f,h):=\nabla_g f \cdot \nabla_g h\\
			\text{Infinitesimal generator:}&\quad Lf:=\Delta_g f -\Gamma(W,f)\\
			\text{Iterated carr\'e du champ:}&\quad \Gamma_2(f):=L\frac{\Gamma(f)}{2}-\Gamma(f,Lf)
		\end{align*}
		where $\Gamma(f):=\Gamma(f,f)$. In addition, $\nabla _g f$ is the Riemannian gradient of $f$ and $\Delta_g f$ is the Laplace-Beltrami operator, i.e. $\Delta_g f=\frac{1}{\sqrt{|g|}}\sum_{i,j}\partial_j(g^{ij}\sqrt{|g|}\partial_i f)$ in a fixed chart where $\partial_i$ denote the standard derivatives in $\R^d$. 
	\end{definition}
	\subsection{Integration by parts} \ \\
	We remark that the definition of $L$ is very natural, since it comes out when considering an integration by parts formula when $M$ has no boundary (to be intended in the sense of differential geometry). Namely, it holds
	\begin{equation}\label{eqippclassic}
		-\int_M h L f d\mu=\int_M \Gamma(h,f)d\mu\qquad\forall f,h \in C_c^\infty(M).
	\end{equation} 
	Let us make now an example where we take $M=\Omega$ as an open set in $\R^d$ equipped with the euclidean metric and a measure $d\mu=e^{-W}dx$. We want to see how \eqref{eqippclassic} changes when, instead of taking $f,h\in C^\infty_c(\Omega)$, we allow $f,h$ to be different from zero on $\partial \Omega$.
	\begin{example}
		Let $\Omega\subset \R^d$ be a piecewise smooth open set. 
		Fix $f\in C^2_c(\overline{\Omega})$ and $h\in C^1(\overline{\Omega})$. Assume that $e^{-W}$ is smooth and strictly positive in \text{supp}$f$. 
		Then,
		\begin{equation}\label{eqipptrue}
			-\int_\Omega h Lf d\mu=\int_\Omega \Gamma (h,f)d\mu- \int_{\partial \Omega} h g^{ij}\partial_j f \nu_i e^{-W}\sqrt{|g|}d\mathcal{H}^{d-1},
		\end{equation} 
		where $\nu$ is the outward vector field of $\partial \Omega$ and $d\mathcal{H}^{d-1}$ is the $d-1$ Hausdorff measure.\\
		Indeed,
		\begin{align*}
			\begin{split}
				-\int_\Omega h Lf d\mu=&-\int_\Omega h\left(\Delta_g f-\Gamma(W,f)\right)e^{-W}\sqrt{|g|}dx\\
				=&-\int_\Omega  h \left(\frac{1}{\sqrt{|g|}}\sum_{i,j}\partial_i(\sqrt{|g|}g^{ij}\partial_j f)-g^{ij}\partial_iW\partial_j f\right)e^{-W}\sqrt{|g|}dx\\
				=&\int_\Omega \left(\partial_i(h e^{-W})g^{ij}\partial_j f-h g^{ij}\partial_i W\partial_j f  e^{-W}\right)\sqrt{|g|}dx-\int_{\partial \Omega} h g^{ij}\partial_j f \nu_i e^{-W}\sqrt{|g|}d\mathcal{H}^{d-1}\\
				=&\int_\Omega g^{ij}\partial_j f \partial_i h \,d\mu- \int_{\partial \Omega} h g^{ij}\partial_j f \nu_i e^{-W}\sqrt{|g|}d\mathcal{H}^{d-1}\\
				=&\int_\Omega \Gamma (h,f)d\mu- \int_{\partial \Omega} h g^{ij}\partial_j f \nu_i e^{-W}\sqrt{|g|}d\mathcal{H}^{d-1}
			\end{split}
		\end{align*}
		where we used integration by parts in $\R^d$ (see Rudin \cite{Rudin}).
	\end{example}
	Formally, if $e^{-W}=0$ on $\partial M$, the second term in the RHS of \eqref{eqipptrue} would be zero. This fact will be proved more rigorously later.\\
	Let us define the weighted manifold where we want to set our problem.\\
	\begin{definition}
		Following \cite{Zug}, define the weighted manifold $(E,g_E,d\mu_E)$ by 
		\begin{equation}\label{defnE}
			E:=\R^d_*,\qquad g_E:=|x|^{2(\alpha-1)}g_{\R^d},\qquad  d\mu_E:=|x|^{-bp}x^A dx=:e^{-W_E}dV_{g_E}
		\end{equation}
		In particular, $W_E=-\log x^A-\frac{\alpha(n-d)-|A|}{2}\log|x|^2$ and so 
		\begin{align*}
			&\Gamma_E(h,f)=|x|^{2(1-\alpha)}\nabla h \cdot \nabla f\\
			&L_E f=|x|^{(1-\alpha)d} \sum_{i=1}^{d}\partial_i (|x|^{(\alpha-1)(d-2)}\partial_i f)+|x|^{2(1-\alpha)}\nabla f \cdot \nabla \left(\log x^A+\frac{\alpha(n-d)-|A|}{2}\log|x|^2\right).
		\end{align*}
		Define also
		\[
		\mathcal{A}_0:=C^\infty_c(\overline{\R^d_*}\setminus \{0\}).
		\]
	\end{definition}
	Suppose from now on that $A\neq 0\in \R^d$ in order to have $0\notin \R^d_*$ \footnote{Everything that follows works the same even if $A=0$ prior defining $E$, and the other spaces, as $E=\R^d\setminus\{0\}$.} .
	\begin{remark}
		Unlike the classical theory of weighted manifolds, we work with $\mathcal{A}_0$ rather that the space $C^\infty_c(\R^d_*)$. As a matter of fact, since \eqref{eqcabros} is valid for functions that do not necessarily vanish on the hyperplanes $\{x_i=0\}$, there is no reason to consider a smaller set of functions than $C^\infty_c(\R^d_*)$. But we need to be cautious with the boundary terms in the integration by parts. We will see later that, via a suitable approximation technique, integrals over $\partial \R^d_*$ in \eqref{eqippE} will not appear since the weight $x^A$ is zero there.
	\end{remark}
	The advantage of this definition is that \eqref{eqmonckn} can be rewritten as a Sobolev's inequality on the space $E$, namely
	\begin{equation}\label{eqsobE}
		\left(\int_E |f|^p d\mu_E\right)^{2/p}\leq C\int_E \Gamma_E(f)d\mu_E, \quad\forall f\in \mathcal{A}_0.
	\end{equation}
	Observe that for \eqref{eqmonckn} it is equivalent to consider functions in $\mathcal{A}_0$ since we integrate over $\R^d_*$.
	We have the following integration by parts formula.
	\begin{proposition}\label{thmipptrue}It holds
		\begin{equation}\label{eqippE}
			-\int_E h L_Ef d\mu_E=\int_E\Gamma_E (h,f)d\mu_E \qquad\forall h\in \mathcal{A}_0,\,\, f \in C^\infty(\R^d\setminus\{0\}).
		\end{equation}
	\end{proposition}
	\begin{proof}
		For simplicity, suppose that $A_1>0$ and $A_i=0$ $\forall i=2,...,d$. The same reasoning can be adapted to the case where $A_i\neq 0$ for some $i$. There exists $R>r>0$ such that $\text{supp}\, h=:\underline{h}\subset \subset \overline{\R^d_*}\cap (B_R\setminus B_r)=:M$. Note that in this case $\R^d_*=\{x_1>0\}$ and define $E_\varepsilon:=\{x_1>\varepsilon\}$.
		Call $\Omega_\varepsilon:=M\cap E_\varepsilon$. For $\varepsilon$ sufficiently small we have that $\partial \Omega_\varepsilon=I_\varepsilon\cup II_\varepsilon\cup III_\varepsilon$ with 
		\begin{align*}
			&I_\varepsilon:= \Omega_\varepsilon\cap \partial  B_R\\
			&II_\varepsilon:= \Omega_\varepsilon\cap \partial B_r\\
			&III_\varepsilon:= \Omega_\varepsilon\cap \{x_1=\varepsilon\}
		\end{align*}

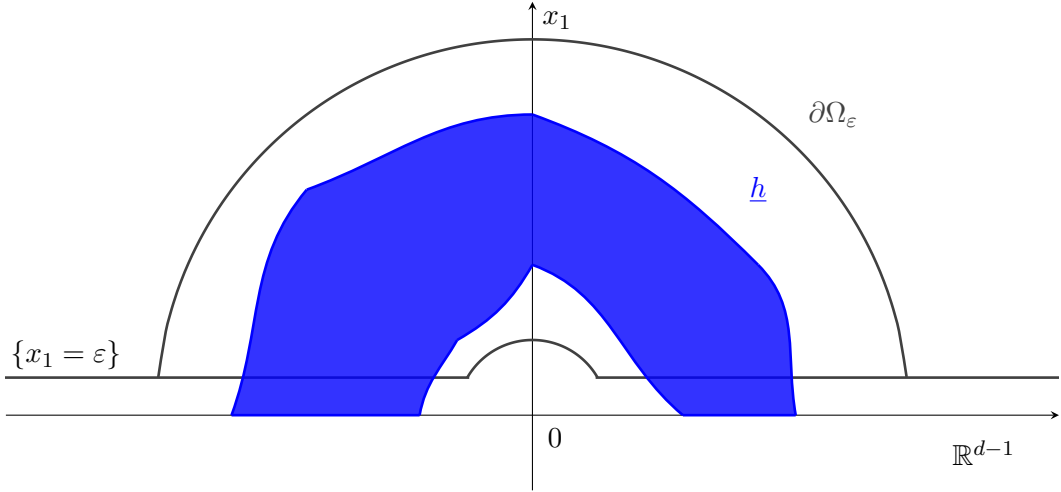
		\begin{figure}[H]
			\begin{tikzpicture}[]
			\begin{axis}[
				x=1cm,y=1cm,
				axis lines=middle,
				xmin=-7,
				xmax=7,
				ymin=-1,
				ymax=5.5,
				xtick=\empty,
				ytick=\empty,
				xlabel={},
				ylabel={$x_1$},
				restrict y to domain=0.5:4,			
				]			
				\draw[line width=1pt,color=darkgray,smooth,samples=100,domain=-0.8660254037844386:0.8660254037844386] plot(\x,{(1-(\x)^(2))^(1/2)});
				\draw[line width=1pt,color=darkgray,smooth,samples=100,domain=-4.9749371855331:4.9749371855331] plot(\x,{(25-(\x)^(2))^(1/2)});
				\draw [line width=1pt,color=darkgray,domain=-12.57:-0.84] plot(\x,{(--0.5-0*\x)/1});
				\draw [line width=1pt,color=darkgray,domain=0.84:12.5] plot(\x,{(--0.5-0*\x)/1});
				\draw[line width=1pt, color=blue] (-1.5,0) -- (-4,0);
				\draw[line width=1pt, color=blue]
				(-4,0) to[out=70,in=230] (-3,3); 
				\draw[line width=1pt, color=blue]
				(-3,3) to[out=20,in=180] (0,4);
				\draw[line width=1pt, color=blue]
				(0,4) to[out=-20,in=135] (3,2);
				\draw[line width=1pt, color=blue]
				(3,2) to[out=-45,in=100] (3.5,0);
				\draw[line width=1pt, color=blue]
				(3.5,0) -- (2,0);
				\draw[line width=1pt, color=blue]
				(2,0) to[out=140,in=-20] (0,2);
				\draw[line width=1pt, color=blue]
				(0,2) to[out=-120,in=30] (-1,1);
				\draw[line width=1pt, color=blue]
				(-1,1) to[out=-120,in=80] (-1.5,0);
				\fill[line width=2pt,color=blue,fill=blue,fill opacity=0.8] 
				(-1.5,0) -- (-4,0) to[out=70,in=230] 
				(-3,3) to[out=20,in=180] 
				(0,4) to[out=-20,in=135] 
				(3,2) to[out=-45,in=100]
				(3.5,0) -- 
				(2,0) to[out=140,in=-20] 
				(0,2) to[out=-120,in=30] 
				(-1,1) to[out=-120,in=80] (-1.5,0)
				;
				
				\node at (4,4) {{\color{darkgray}$\partial \Omega_\varepsilon$}};
				
				\node at (3,3)
				{{\color{blue}$\underline{h}$}};
				
				\node at (-6.2,0.8) {$\{x_1=\varepsilon \}$};
				\node at (6,-0.5)
				{$\R^{d-1}$};
				
				\node at (0.3,-0.3)
				{$0$};
			\end{axis}	
		\end{tikzpicture}
				\end{figure}

		Since in $\Omega_\varepsilon$ we have $e^{-W_E}$ strictly positive and smooth, we can use \eqref{eqipptrue} and deduce that
		\begin{align}\label{eqapprox}
			\begin{split}
				-\int_{\Omega_\varepsilon} h L_Ef d\mu_E&=-\int_{\partial \Omega_\varepsilon} h g^{ij}\partial_j f \nu_i e^{-W_E}\sqrt{|g|}d\mathcal{H}^{d-1}+\int_{\Omega_\varepsilon} \Gamma_E(h,f)d\mu_E\\
				&=-\int_{III_\varepsilon} h g^{ij}\partial_j f \nu_i e^{-W_E}\sqrt{|g|}d\mathcal{H}^{d-1}+\int_{\Omega_\varepsilon} \Gamma_E(h,f)d\mu_E
			\end{split}
		\end{align}
		since $\underline{h}\subset \subset  B_R\setminus B_r$.
		It can be easily verified that $hL_E f, \Gamma_E(h,f)\in L^1(M,d\mu_E)$. Indeed, $\Gamma_E(h,f)$ is bounded and compactly supported in $M$. Similarly, the terms of $hL_E f$ are all bounded in $M$ exept for the second term which is bounded by $c\frac{1}{x_1}$ while $d\mu_E\leq c x_1^{A_1}dx$ in $M$ and so $hL_E f \in L^1(M,d\mu_E)$. Similarly, we deduce
		\begin{align*}
			\left|\int_{III_\varepsilon} h g^{ij}\partial_j f \nu_i e^{-W_\varepsilon}\sqrt{|g|}d\mathcal{H}^{d-1}\right|\leq c \varepsilon^{A_1}
		\end{align*}
		so we can pass to the limit as $\varepsilon\rightarrow 0$ and prove \eqref{eqippE}.
	\end{proof}
	By symmetry, \eqref{eqippE} holds if $f\in \mathcal{A}_0$ and $h\in C^\infty(\R^d\setminus\{0\})$. \\
	Even though the weighted manifold $E$ provides a pleasant expression for the Sobolev inequality \eqref{eqmonckn}, its measure $d\mu_E$ is not finite. As a consequence, we define another weighted manifold equipped with a finite measure in order to take advantage of compactness results (see \eqref{eqrellich} later). Call 
	\[
	\varphi(x):=\frac{1+|x|^{2\alpha}}{2},\quad x\in \R^d.
	\]
	\begin{definition}
		Define the weighted manifold $(S,g_S, d\mu_S)$ as 
		\[
		S=E=\R^d_*,\qquad g_S:=\varphi^{-2}g_E, \qquad d\mu_S=\varphi^{-n}d\mu_E.
		\]
	\end{definition}
	Observe that the manifold $S$ is topologically equal to $E$, but equipped with a metric $g_S$ and measure $d\mu_S$ obtained by multiplying $g_E$, $d\mu_E$ by suitable powers of the conformal factor $\varphi$.  It is therefore easy to verify that, under \eqref{eqhpab}, we have $\mu_S(S)<+\infty$. Indeed,
	\begin{align*}
		&\int_{\R^d_*\setminus B_1}d\mu_S\leq  c\int_{\R^d\setminus B_1} |x|^{-bp+|A|-2\alpha n}dx=c\int_1^{+\infty} r^{-bp+D-1-2\alpha n}dr\overset{\eqref{eq1.12}}{=} c\int_1^{+\infty} r^{-1-\alpha n}dr<+\infty\\
		&\int_{B_1}d\mu_S\leq c \int_{B_1}  |x|^{-bp +|A|}dx=c \int_{0}^1r^{-bp+D-1}dr\overset{\eqref{eq1.12}}{=}c\int_0^1 r^{-1+\alpha n}dr <+\infty,
	\end{align*}
	where we used that $\alpha>0$ since we suppose \eqref{eqhpab}. As for $E$, the integration by parts formula holds also in $S$. Namely
	\begin{proposition}It holds
		\begin{equation}\label{eqipp}
			-\int_S h L_Sf d\mu_S=\int_S\Gamma_S (h,f)d\mu_S \qquad\forall h\in \mathcal{A}_0,\,\, f \in C^\infty(\R^d\setminus\{0\}).
		\end{equation}
	\end{proposition}
	\begin{proof}
		The proof is the same as that of Proposition \ref{thmipptrue}.
	\end{proof}
	Our goal is to get to a Sobolev inequality on $S$ that is equivalent to the one in $E$. Let us for this reason recall the concept of $n$-conformality introduced in \cite{BakryGentilLedoux} by Bakry, Gentil, Ledoux. 
	\subsection{n-conformal invariance}
	\begin{definition}
		$(M,\tilde{g},\tilde{\mu})$ is n-conformal to $(M,g,\mu)$ if there exists $\psi \in C^\infty(M)$, $\psi>0$ such that 
		\begin{align*}
			\tilde{g}=\psi^2 g\qquad d\tilde{\mu}=\psi^n d\mu
		\end{align*}
	\end{definition}
	Thus, $E$ and $S$ are $n$-conformal by construction with 
	\begin{align*}
		\tilde{g}=g_S,\, d\tilde \mu=d\mu_S, \quad g=g_E,\,d\mu=d\mu_E\quad\text{and }\psi=\varphi^{-1}.
	\end{align*}
	Follows a result of invariance for Sobolev inequality under $n$-conformal transformations.
	\begin{proposition}\label{sobinvariance}
		Let $(M,\tilde{g},\tilde{\mu})$ be n-conformal to $(M,g,\mu)$ and let $L$ be the infinitesimal generator associated to $(M,g,\mu)$.
		Assume there exists $V\in C^\infty(M)$ such that Sobolev's inequality holds in the form
		\begin{align*}
			\left(\int |f|^p d\mu\right)^{2/p}\leq C\left(\int |\nabla _g f|^2 d\mu+\int V f^2 d\mu\right),\quad p=\frac{2n}{n-2}\qquad\forall f\in C^\infty_c(M).
		\end{align*}
		Then, with the same constant $C$,
		\begin{align*}
			\left(\int|F|^p d \tilde{\mu}\right)^{2/p}\leq C\left(\int|\nabla_{\tilde{g}}F|^2 d\tilde{\mu}+\int \tilde{V} F^2 d\tilde{\mu}\right)\quad\forall F \in C^\infty_c(M)
		\end{align*}
		where $\tilde{V}$ is computed through Yamabe's equation
		\begin{align*}
			\begin{cases}
				-L u + V u= \tilde{V} u^{p-1}\qquad \text{in }M\\
				u=\psi^\frac{n-2}{2}
			\end{cases}
		\end{align*}
	\end{proposition}
	\begin{proof}
		Recall that for weighted manifolds the following integration by parts formula holds (see \cite{Zug})
		\begin{align}\label{ippsenzabordo}
			-\int_M h Lf d\mu=\int_M \nabla_g h \cdot  \nabla_g fd\mu\qquad\forall f,h\in C^\infty_c(M).
		\end{align}
		Fix $f\in C^\infty_c(M)$ and define $F:=fu^{-1}$. A direct computation shows that 
		\begin{align*}
			\int |f|^p d\mu=\int |F|^p d\tilde{\mu}\\
			\text{and }\quad\int V f^2 d\mu=\int V F^2 u^{2-p}d\tilde{\mu}.
		\end{align*}
		In addition,
		\begin{align*}
			\int |\nabla _g f|^2 d\mu&=\int |\nabla_g (uF)|^2 d\mu=\int u^2|\nabla _g F|^2d\mu+\int \nabla _g(F^2 u)\cdot \nabla _g u\, d\mu\\
			&\overset{\eqref{ippsenzabordo}}{=}\int |\nabla _{\tilde{g}}F|^2 d\tilde{\mu}-\int Lu \,\,u F^2 d\mu=\int |\nabla _{\tilde{g} }F|^2 d\tilde{\mu}+\int (-Lu)u^{1-p}F^2 d\tilde{\mu}
		\end{align*}
		and the thesis follows.
	\end{proof}
	Thanks to this result, we can write the Sobolev inequality for $S$. It can be easily computed that 
	\[
	-L_E u=\frac{\alpha^2n(n-2)}{4} u^{p-1}.
	\]
	Eventually, keeping the notation of Proposition \ref{sobinvariance} we get $\tilde{V}=\frac{\alpha^2n(n-2)}{4} $, $V=0$ and so \eqref{eqsobE} is equivalent to 
	\begin{equation}\label{eqsobS}
		\left(\int_S |F|^p d\mu_S\right)^{2/p}\leq C\left(\int_S \Gamma_S(F)d\mu_S+\frac{n(n-2)}{4}\alpha^2\int_S F^2 d\mu_S\right) \qquad\forall F\in \mathcal{A}_0
	\end{equation}
	\begin{remark}
		To be precise, Proposition \ref{sobinvariance} should be applied considering the space of functions $C^\infty_c(\R^d_*)$ rather than $\mathcal{A}_0$ in our models $E$ and $S$. However, since the integration by parts formulas \eqref{eqippE} and \eqref{eqipp} are valid for functions that do not necessarily vanish on the boundary, the same proof of Proposition \ref{sobinvariance} works considering this larger set of functions. Therefore, we really have equivalence between \eqref{eqsobE} and \eqref{eqsobS} for functions in $\mathcal{A}_0$ and not only in $C^\infty_c(\R^d_*)$.
	\end{remark}
	For this reason, the optimal constant in \eqref{eqsobS} is the optimal constant for \eqref{eqsobE}. Therefore, from now on we work with $S$ and the goal is to compute the optimal contant of \eqref{eqsobS}.\\
	In the following subsection, we provide a representation of the space $S$ as a warped space. This is particularly useful because it allows us to use some curvature-dimension condition results computed in \cite{DSZ}.
	\subsection{Warped products}\label{appwarped}
	\
	\\
	Suppose from now on $A_d=0$. Then, the spaces $E$ and $S$ can be seen as warped products; see Lee \cite{Lee} for the definition of warped products. Indeed, let 
	\begin{equation}\label{monspheredef}
		\mathbb{S}^{d-1}_*:=\{\theta=(\theta_1,...,\theta_d)\in\mathbb{S}^{d-1}\subset \R^d, \text{for all }i=1,...,d;\,\,\theta_i>0 \text{ whenever }A_i>0 \}
	\end{equation}
	and equip $	\mathbb{S}^{d-1}_*$ with the standard round metric $g_{\mathbb{S}^{d-1}}$ and the weighted measure $d\mu_{\mathbb{S}^{d-1}_*}:=\theta^A dV_{\mathbb{S}^{d-1}}=:e^{-W_\theta} dV_{\mathbb{S}^{d-1}}$. Let $L_\theta$ be its infinitesimal generator, $\Gamma^\theta$ its carr\'e du champ and $\Gamma^\theta_2$ its iterated carr\'e du champ.\\
	Then, if we set $\rho=r^\alpha=|x|^\alpha$ and $\alpha^2=1/(1+\gamma^2)$, our model $E$ can be written as the cone
	\[
	E=\{(\rho \theta, \gamma \rho): \rho \in \R_+\setminus\{0\}, \theta\in \mathbb{S}^{d-1}_* \}
	\]
	equipped with metric and measure
	\begin{align*}
		\begin{split}
			g_E=\frac{1}{\alpha^2}d\rho^2+\rho^2 g_{\mathbb{S}^{d-1}}, \\ d\mu_E=\frac{1}{\alpha}\rho^{n-1}d\rho d\mu_{\mathbb{S}^{d-1}_*}.
		\end{split}
	\end{align*}
	Following \cite{Zug}, this can be seen via the classical change of measure $dx=r^{d-1}dr dV_{\mathbb{S}^{d-1}}$ and noting that $x^A=r^{|A|}\theta^A$. Therefore, if we equip $\R_+\setminus \{0\}$ with the metric $g=\frac{1}{\alpha^2}d\rho^2$, we obtain $E=(\R_+\setminus \{0\})\times_f \mathbb{S}^{d-1}_*$ with warping function $f=\rho^2$.\\
	We now want to show that also $S$ can be seen as a warped product 
	\begin{equation}\label{eqwarped}
		S=I\times_f \mathbb{S}^{d-1}_*,
	\end{equation}
	with $I=(-1,1)$, the metric $g_I:=\frac{1}{\alpha^2(1-y^2)^2}dy^2$ and with warping function $f(y)=(1-y^2)$, $y\in I$. Indeed, perform the change of variables $y=\frac{|x|^{2\alpha}-1}{|x|^{2\alpha}+1}=\frac{\rho^2-1}{\rho^2+1}\in I$ and observe that
	\begin{equation*}
		g_S=(1-y)^2g_E=\frac{1}{\alpha^2}\frac{1}{1-y^2}dy^2+(1-y^2)g_{\mathbb{S}^{d-1}}.
	\end{equation*}
	In addition, we have 
	\begin{align}\label{eqLwarpd}
		\begin{split}
			&d\mu_S=(1-y)^nd\mu_E=\frac{(1-y^2)^{\frac{n}{2}-1}}{\alpha}dyd\mu_{\mathbb{S}^{d-1}_*}\\
			&\Gamma_S(f)=\alpha^2(1-y^2)(\partial_y f)^2+\frac{1}{1-y^2}\Gamma^\theta (f),\\
			&L_S f=\alpha^2[(1-y^2)\partial_{yy}f-ny\partial_y f]+\frac{1}{1-y^2}L_\theta f.
		\end{split}
	\end{align}
	Now we can introduce the functional setting and prove some theorems of approximation. We will see that in the variables $(y,\theta)$ the notation is simplified.
	\subsection{Functional setting and approximation results}\ \\
	Define now $L^q(S):=L^q(\R^d_*, d\mu_S)$ for $q\geq 1$ and $H^1_0(S)$ as the closure of $\mathcal{A}_0$ with respect to the norm 
	\[
	\norm{u}_{H^1_0(S)}:=\left(\int_S u^2 d\mu_S+\int_S \Gamma_S(u)d\mu_S\right)^{1/2}.
	\] 
	It can be proved, as done in \cite[Lemma A.4]{Zug}, that the embedding 
	\begin{equation}\label{eqrellich}
		H^1_0(S)\hookrightarrow L^q(S)
	\end{equation}
	is compact for $q\in[1,p)$ since the space $S$ has finite measure and a Sobolev inequality (inequality \eqref{eqsobS}) is valid with exponent $p$.\\
	At last, define the space 
	\begin{equation*}
		D(L_S)=\{v\in H^1_0(S)\,:\,L_S v\in L^2(S)\}
	\end{equation*}
	where $L_Sv$ has to be intended as a (weighted) distribution, and equip $D(L_S)$ with the norm
	\begin{equation}\label{eqnormDL}
		\norm{v}_{D(L_S)}^2:=\norm{v}_{L^2(S)}^2+ \norm{L_S v}_{L^2(S)}^2
	\end{equation} 
	Let us now prove some approximation results that will be useful later on.
	\begin{proposition}\label{prop1h0}
		Let $n>2$. Then, the constant function $\mathbf{1}$ belongs to $H^1_0(S)$ and $\Gamma_S(\mathbf{1})=0$. Moreover, if $n>4$, there exists a sequence of functions $\zeta_k\in \mathcal{A}_0$ which approximates $\mathbf{1}$ in $H^1_0(S)$ such that $L_S(\zeta_k)\rightarrow 0$ in $L^2(S)$ .
	\end{proposition}
	\begin{proof}
		Fix $k\in \N$ and define $\zeta_k$ as a smooth function on $I=(-1,1)$ such that 
		\begin{align*}
			\zeta_k(y)=
			\begin{cases}
				0\quad \text{if }y\in (-1,-1+1/k)\cup (1-1/k,1)\\
				1\quad \text{if }y\in (-1+2/k,1-2/k)
			\end{cases}
		\end{align*}
		and $|\partial_y \zeta_k|\leq 2k$, $|\partial_{yy} \zeta_k|\leq 2 k^2$. 
		From now on see $\zeta_k$ as a radial function defined in the whole space $S=I\times_f \mathbb{S}^{d-1}_*$. By definition, $\zeta_k\in \mathcal{A}_0$. and
		by dominated convergence theorem, $\zeta_k \rightarrow\mathbf{1}$ in $L^2(S)$. Let us now evaluate $\Gamma_S(\zeta_k)$. From the representation \eqref{eqwarped} we deduce 
		\begin{align*}
			\int_S \Gamma_S(\zeta_k)d\mu_S&\leq c\left(\int_{-1+1/k}^{-1+2/k}|\partial_y \zeta_k|^2 (1-y^2)^\frac{n}{2}dy+ \int_{1-2/k}^{1-1/k}|\partial_y \zeta_k|^2 (1-y^2)^\frac{n}{2}dy\right)\\
			&\leq c \int _0^{2/k}|\partial_y \zeta_k|^2 t^\frac{n}{2} dt \leq  c k^2 \left(\frac{1}{k}\right)^\frac{n+2}{2}\xrightarrow{k\rightarrow +\infty}0
		\end{align*}
		where we performed the change of variables $t=1+y$ in the first integral and $t=1-y$ in the second. It tells us that $\mathbf{1}\in H^1_0(S)$ with $\Gamma_S(\mathbf{1})=0$. 
		Let us now evaluate $\norm{L_S \zeta_k}_{L^2(S)}$ under the hypothesis $n>4$.
		\begin{align*}
			\int_S L_S(\zeta_k)^2d\mu_S \leq c \left( \int_0^{2/k}|\partial_{yy} \zeta_k|^2 t^{\frac{
					n}{2}+1}dt+ \int _0^{2/k}|\partial _y \zeta_k|^2 t^{\frac{n}{2}-1}dt\right)\leq c k^4 \left(\frac{1}{k}\right)^\frac{n+4}{2}\xrightarrow{k\rightarrow +\infty}0,
		\end{align*}
		proving the last part of the proposition.
	\end{proof}
	In the following sections it will be necessary to approximate functions under the $D(L_S)$ norm introduced in \eqref{eqnormDL}. This is the reason why we have this proposition.
	\begin{proposition}\label{lemA.2}
		Assume $n>4$. Let $f$ such that $\Gamma_S(f),L_S f\in L^2(S)$ and $f\in L^\infty$. Then, there exists $f_k\in \mathcal{A}_0$ such that $f_k\rightarrow f$ in $D(L_S)$. 
	\end{proposition}
	\begin{proof}
		We define $f_k:=f \zeta_k$ with $\zeta_k$ given by Proposition \ref{prop1h0}. We can check (in the variables $(y,\theta)$) that 
		\begin{align*}
			\int_S |f-f_k|^2 d\mu_S\rightarrow 0\qquad\int_S (L_S(f-f_k))^2 d\mu_S\rightarrow 0
		\end{align*}
		Indeed, the first limit holds by dominated convergence theorem. For the second limit, we write
		\begin{align*}
			\int_S(L_S(f-f_k))^2 d\mu_S\leq& 2\int_S(L_S f)^2(\zeta_k-1)^2d\mu_S +4\int_S\Gamma_S(f,\zeta_k)^2d\mu_S +2\int_S f^2(L_S \zeta_k)^2d\mu_S\\
			=&I_1+I_2+I_3
		\end{align*}
		$I_1\rightarrow0 $ by dominated convergence theorem, $I_2\rightarrow 0$ since $\Gamma_S(f)\in L^2$ and $\Gamma_S(\zeta_k)\rightarrow 0$ in $L^2$. Indeed,
		\begin{align*}
			\int_S\Gamma_S(f,\zeta_k)^2 d\mu_S\overset{\text{Cauchy-Schwarz}}{\leq } \int_S \Gamma_S(f)\Gamma_S(\zeta_k)d\mu_S\overset{\text{H\"older}}{\leq }\left(\int_S\Gamma_S(f)^2d\mu_S\right)^\frac{1}{2}\left(\int_S \Gamma_S(\zeta_k)^2d\mu_S\right)^\frac{1}{2}
		\end{align*}
		and
		\begin{align*}
			\int_S \Gamma_S(\zeta_k)^2d\mu_S\leq c k^4 \int_0^{2/k}t^{\frac{n+2}{2}}dt\leq c k^{\frac{4-n}{2}}\rightarrow 0
		\end{align*}
		Eventually,
		\[
		I_3\leq C \|f\|^2_\infty k^{\frac{4-n}{2}}\rightarrow 0
		\]
		proving the thesis.
	\end{proof}
	To conclude this section, we prove a lemma which  concerns the iterated carr\'e du champ. It is an application of the integration by parts under certain hypotheses of regularity. 
	\begin{lemma}\label{corgam2}
		Assume $f\in C^3_c(\overline{\R^d_*}\setminus\{0\})$, and $L_S f\in C_c^1(\overline{\R^d_*}\setminus\{0\})$. Then,
		\begin{equation}\label{eqgam2}
			\int_S\Gamma^2_S (f)d\mu_S=\int_S (L_S f)^2 d\mu_S.
		\end{equation} 
		
	\end{lemma}
	\begin{proof}
		Being $\mathbf{1}\in C^\infty(\R^d)$ and $\Gamma_S(f)\in C^2_c( \overline{\R^d_*}\setminus\{0\})$, we can apply \eqref{eqipp} and deduce that 
		\begin{equation*}
			\int_S L_S(\Gamma_S(f))d\mu_S=-\int_S \Gamma_S(1,\Gamma_S(f))d\mu_S=0.
		\end{equation*}
		As a consequence, \eqref{eqgam2} becomes 
		\begin{equation*}
			-\int_S \Gamma_S(f,L_S f)d\mu_S=\int_S(L_S f)^2 d\mu_S
		\end{equation*}
		which is true applying again \eqref{eqipp}.
	\end{proof}
	Before delving into the main theorem, we still need to deduce some properties for the space we consider. Indeed, we need to clarify an integration by parts formula for $\mathbb{S}^{d-1}_*$ and prove a curvature bound under the generalized Felli-Schneider condition. Namely, in the following section we prove an integrated version of the curvature dimension condition \eqref{FScond}, that is an integral inequality involving $\Gamma_2^S$, $\Gamma_S$ and $L_S$.
	\subsection{Properties of $\mathbb{S}^{d-1}_*$ and integrated curvature dimension condition}\label{monsphere}\ \\
	Suppose that at least one entry in $A\in \R^d$ is equal to zero. Be it without loss of generality $A_d=0$. We then consider the Monomial Sphere $\mathbb{S}^{d-1}_*$ as defined in \eqref{monspheredef} and we prove the following integration by parts formula.
	\begin{proposition}\label{thmippsfmon}
		Let $f\in C^1(\mathbb{S}^{d-1})$ and $h\in C^2(\mathbb{S}^{d-1})$. Then,
		\begin{equation*}
			-\int_{{\mathbb{S}^{d-1}_*}}f L_\theta h d\mu_{\mathbb{S}^{d-1}_*}=\int_{{\mathbb{S}^{d-1}_*}} \Gamma^\theta(f,h)d\mu_{\mathbb{S}^{d-1}_*}.
		\end{equation*}
	\end{proposition}
	\begin{proof}
		For symplicity, assume $A_1>0$ and $A_i=0$ for all $i=2,...d$. Assume also that $f$ is supported in $\mathbb{S}^{d-1}\setminus\{N\}$. The stereographic projection and dominated convergence theorem tell us that 
		\begin{align*}
			-\int_{{\mathbb{S}^{d-1}_*}}f L_\theta h d\mu_{\mathbb{S}^{d-1}_*}=-\int_{{\mathbb{\R}^{d-1}_*}}f L_\theta h e^{-W_\theta}\sqrt{|g_{\mathbb{S}^{d-1}}|}dz=-\lim_{\varepsilon\rightarrow0}\int_{\{z_1\geq \varepsilon\}} f L_\theta h e^{-W_\theta}\sqrt{|g_{\mathbb{S}^{d-1}}|}dz
		\end{align*}
		Similarly to the proof of Theorem \ref{thmipptrue}, there exists $R>0$ such that $\text{supp} f=:\underline{f}\subset \subset \overline{\R^{d-1}_*}\cap B_R=:M$, where $B_R$ is the ball centered in $0$ in $\R^{d-1}$. 
		Call $M_\varepsilon:=M\cap \{z_1\geq\varepsilon\}$. For $\varepsilon$ sufficiently small we have that $\partial M_\varepsilon=I_\varepsilon\cup II_\varepsilon$ with 
		\begin{align*}
			&I_\varepsilon:= M_\varepsilon\cap \partial  B_R\\
			&II_\varepsilon:= M_\varepsilon\cap \{z_1=\varepsilon\}
		\end{align*}
		Thus, using \eqref{eqipptrue} we deduce
		\begin{align*}
			\begin{split}
				-\int_{M_\varepsilon} f L_\theta h e^{-W_\theta}\sqrt{|g_{\mathbb{S}^{d-1}}|}dz&=-\int_{\partial M_\varepsilon} f g_{\mathbb{S}^{d-1}}^{ij}\partial_j h \nu_i e^{-W}\sqrt{|g_{\mathbb{S}^{d-1}}|}d\mathcal{H}^{d-1}+\int_{M_\varepsilon} \Gamma^\theta(f,h)e^{-W_\theta}\sqrt{|g_{\mathbb{S}^{d-1}}|}dz\\
				&=-\int_{II_\varepsilon} f g_{\mathbb{S}^{d-1}}^{1j}\partial_j h \nu_1 e^{-W}\sqrt{|g_{\mathbb{S}^{d-1}}|}d\mathcal{H}^{d-1}+\int_{M_\varepsilon} \Gamma^\theta(f,h)e^{-W_\theta}\sqrt{|g_{\mathbb{S}^{d-1}}|}dz
			\end{split}
		\end{align*}
		However, 
		\begin{align*}
			\left|\int_{II_\varepsilon} f g^{1j}_{\mathbb{S}^{d-1}}\partial_j h \nu_1 e^{-W}\sqrt{|g_{\mathbb{S}^{d-1}}|}d\mathcal{H}^{d-1}\right|\leq c \varepsilon^{A_1},\\ L_\theta h  \in L^1(M,e^{-W_\theta}\sqrt{|g_{\mathbb{S}^{d-1}}|}dz),\qquad \Gamma^\theta(f,h)\in L^1(M,e^{-W_\theta}\sqrt{|g_{\mathbb{S}^{d-1}}|}dz) 
		\end{align*}
		so we can pass to the limit as $\varepsilon\rightarrow 0$, proving the thesis if $\underline{f}\subseteq \mathbb{S}^{d-1}_*\setminus\{N\}$. Clearly the same reasoning applies also if $\underline{f}\subseteq \mathbb{S}^{d-1}_*\setminus\{S\}$. Eventually, to prove the general case, we consider a partition of unity $\eta_i$, $i=1,2$ subordinate to $\mathbb{S}^{d-1}\setminus\{N\}\cup \mathbb{S}^{d-1}\setminus\{S\}$. Indeed,
		\begin{align*}
			\int_{{\mathbb{S}^{d-1}_*}} fL_\theta h d\mu_{\mathbb{S}^{d-1}_*}&=-\sum_{i=1}^2\int_{{\mathbb{S}^{d-1}_*}} \eta_i f L_\theta h d\mu_{\mathbb{S}^{d-1}_*}=\sum_{i=1}^2\int_{{\mathbb{S}^{d-1}_*}} \Gamma^\theta(\eta_i f,h) d\mu_{\mathbb{S}^{d-1}_*}\\
			&=\int_{{\mathbb{S}^{d-1}_*}}\Gamma^\theta(\sum_{i=1}^2\eta_i f,h)d\mu_{\mathbb{S}^{d-1}_*}=\int_{{\mathbb{S}^{d-1}_*}}\Gamma^\theta\left( f,h\right)d\mu_{\mathbb{S}^{d-1}_*}
		\end{align*}
	\end{proof}
	Lastly, as proved later in paragraph \ref{subseccdsphere}, $\mathbb{S}^{d-1}_*$ satisfies a $CD(D-2,D-1)$ condition, namely that 
	\begin{equation}\label{eqcdsp}
		\Gamma_2^\theta(f)-(D-2)\Gamma^\theta(f)-\frac{1}{D-1}(L_\theta f)^2\geq 0 \qquad\forall f\in C^\infty(\mathbb{S}^{d-1}_*).
	\end{equation}
	We now have all the ingredients to prove the integrated curvature dimension condition. 
	\begin{theorem}\label{thmintcd}
		Under the hypothesis 
		\begin{equation}\label{hpalphaint}
			\alpha^2\leq \frac{	D-1}{n-1}
		\end{equation}
		$S$ satisfies the integrated $CD(\rho,n)$ condition, with $\rho =\alpha^2(n-1)$. Namely, for all $f$  such that
		\begin{itemize}
			\item[(i)]$f> 0$
			\item[(ii)] $f\in C^3(\overline{\R^d_*}\setminus\{0\})$
			\item[(iii)] $L_S f \in C^1(\overline{\R^d_*}\setminus\{0\})$
		\end{itemize}
		and $\nu\in \R$, $\nu>n$, it holds
		\begin{equation}\label{eqintcd}
			\int _{\mathbb{S}^{d-1}_*} \left(\Gamma_2^S(f)-\alpha^2(n-1)\Gamma_S(f)-\frac{1}{n}(L_Sf)^2\right)f^{1-\nu} d\mu_{\mathbb{S}^{d-1}_*} \geq 0 
		\end{equation}
	\end{theorem}
	\begin{proof}
		Viewing $S$ as the warped product \eqref{eqwarped}, the following formula can be proved (see also \cite{DSZ})		
		\begin{align}\label{eq5.59}
			\begin{split}
				\Gamma_2^S(f)-\alpha^2(n-1)\Gamma_S(f)-\frac{1}{n}(L_Sf)^2&=\frac{1}{n}\left(\alpha^2(1-y^2)\sqrt{n-1}\partial_{yy}f-\frac{1}{\sqrt{n-1}}\frac{L_\theta f}{1-y^2}\right)^2\\&+2\alpha^2\Gamma^\theta\left(\frac{fy}{1-y^2}+\partial_y f\right)\\
				& +\frac{1}{(1-y^2)^2}\left(\Gamma^\theta_2(f)-\alpha^2(n-2)\Gamma^\theta(f)-\frac{1}{n-1}(L_\theta f)^2\right).
			\end{split}
		\end{align}
		for $f\in C^3$.
		Fix $y\in (-1,1)=I$. $(iii)$ implies that $$\norm{L_Sf(y,\cdot)}_{C^1(\mathbb{S}^{d-1}_*)}\leq c_y$$ and $(ii)$ implies in turn that $$\norm{\partial_y f(y,\cdot)}_{C^1(\mathbb{S}^{d-1}_*)}+\norm{\partial_{yy} f(y,\cdot)}_{C^1(\mathbb{S}^{d-1}_*)}\leq c_y$$ for a constant $c_y$ depending on $y$. Thus, \eqref{eqLwarpd} tells us that 
		\[
		\norm{L_\theta f(y,\cdot)}_{C^1(\mathbb{S}^{d-1}_*)}\leq c_y.
		\]
		This last bound allows us to apply Proposition \ref{thmippsfmon} and have 
		\[
		-\int_{{\mathbb{S}^{d-1}_*}}\Gamma^\theta(f,L_\theta f)d\mu_{\mathbb{S}^{d-1}_*}=\int_{{\mathbb{S}^{d-1}_*}}(L_\theta f)^2d\mu_{\mathbb{S}^{d-1}_*}
		\]
		that, together with
		\[
		\int_{{\mathbb{S}^{d-1}_*}}L_\theta \Gamma^\theta( f) d\mu_{\mathbb{S}^{d-1}_*}=0
		\]
		since $\mathbf{1}\in C^\infty(\mathbb{S}^{d-1})$, 
		we get 
		\begin{equation}\label{eqsL=G2}
			\int_{{\mathbb{S}^{d-1}_*}} \Gamma^\theta_2(f)d\mu_{\mathbb{S}^{d-1}_*}=\int_{{\mathbb{S}^{d-1}_*}}(L_\theta f)^2d\mu_{\mathbb{S}^{d-1}_*}.
		\end{equation}
		This last equality together with \eqref{eqcdsp} allows us to apply \cite[p. 767]{GentZug} and deduce
		\begin{equation}\label{eq5.60b}
			\int _{\mathbb{S}^{d-1}_*} \Gamma^\theta_2(f)f^{1-\nu}d\mu_{\mathbb{S}^{d-1}_*}\geq (D-1)\int_{\mathbb{S}^{d-1}_*} \Gamma^\theta (f)f^{1-\nu}d\mu_{\mathbb{S}^{d-1}_*}+B\int_{\mathbb{S}^{d-1}_*} \frac{\Gamma^\theta(f)^2}{f^2}f^{1-\nu}d\mu_{\mathbb{S}^{d-1}_*},
		\end{equation}
		with $B=B(\nu,D)>0$, which yields in particular 
		\begin{equation}\label{eq5.60}
			\int _{\mathbb{S}^{d-1}_*} \Gamma^\theta_2(f)f^{1-\nu}d\mu_{\mathbb{S}^{d-1}_*}\geq (D-1)\int _{\mathbb{S}^{d-1}_*} \Gamma^\theta (f)f^{1-\nu}d\mu_{\mathbb{S}^{d-1}_*}.
		\end{equation}
		
		Also, \eqref{eqcdsp} trivially implies
		\begin{equation}\label{eq5.61}
			-\frac{(L_\theta f)^2}{n-1}\geq \frac{D-1}{n-1}(-\Gamma^\theta_2(f))+\frac{(D-2)(D-1)}{n-1}\Gamma^\theta(f).
		\end{equation}
		Eventually,
		\begin{align}\label{eq5.62}
			\begin{split}
				&\int _{\mathbb{S}^{d-1}_*} \left(\Gamma_2(f)-\alpha^2(n-1)\Gamma_S(f)-\frac{1}{n}(L_Sf)^2 \right)f^{1-\nu}d\mu_{\mathbb{S}^{d-1}_*}\\&\overset{\eqref{eq5.59}}{\geq} \frac{1}{(1-y^2)^2}\int_{\mathbb{S}^{d-1}_*}\left(\Gamma^\theta_2(f)-\alpha^2(n-2)\Gamma^\theta(f)-\frac{1}{n-1}(L_\theta f)^2\right)f^{1-\nu}d\mu_{\mathbb{S}^{d-1}_*}\\
				&\overset{\eqref{eq5.61}}{\geq}\frac{1}{(1-y^2)^2}\int_{\mathbb{S}^{d-1}_*}\left( \frac{n-D}{n-1}\Gamma_2^\theta(f)+\frac{(D-2)(D-1)-\alpha^2(n-2)(n-1)}{n-1}\Gamma^\theta(f)\right)f^{1-\nu}d\mu_{\mathbb{S}^{d-1}_*}\\
				&\overset{\eqref{eq5.60}}{\geq}\frac{1}{(1-y^2)^2}\int_{\mathbb{S}^{d-1}_*} \left(\frac{(n-D)(D-1)+(D-2)(D-1)-\alpha^2(n-2)(n-1)}{n-1}\right)\Gamma^\theta (f)f^{1-\nu}d\mu_{\mathbb{S}^{d-1}_*}\\
				&=\frac{1}{(1-y^2)^2}\frac{n-2}{n-1}\left((D-1)-\alpha^2(n-1)\right)\int_{\mathbb{S}^{d-1}_*} \Gamma^\theta(f)f^{1-\nu}d\mu_{\mathbb{S}^{d-1}_*}\geq 0.
			\end{split}
		\end{align}
		The thesis follows.
	\end{proof}
	\begin{remark}
		Observe that, if we used \eqref{eq5.60b} in place of \eqref{eq5.60} in the last estimate, we would get to 
		\begin{align}
			\begin{split}\label{eq5.62b}
				&\int _{\mathbb{S}^{d-1}_*} \left(\Gamma_2(f)-\alpha^2(n-1)\Gamma_S(f)-\frac{1}{n}(L_Sf)^2 \right)f^{1-\nu}d\mu_{\mathbb{S}^{d-1}_*}\\
				&\geq  \frac{1}{(1-y^2)^2(n-1)}\Bigg[(n-2)\left((D-1)-\alpha^2(n-1)\right)\int_{\mathbb{S}^{d-1}_*} \Gamma^\theta(f)f^{1-\nu}d\mu_{\mathbb{S}^{d-1}_*}\\&+ B(n-D)\int_{\mathbb{S}^{d-1}_*} \frac{\Gamma^\theta(f)^2}{f^2}f^{1-\nu} d\mu_{\mathbb{S}^{d-1}_*} \Bigg]
			\end{split}
		\end{align}
	\end{remark}
	To conclude, it follows a theorem that allows us to prove the positivity of the operator $\int_{\mathbb{S}^{d-1}_*} \Gamma^S_2(\cdot)d\mu_{\mathbb{S}^{d-1}_*}$ for each $y$ fixed.
	\begin{theorem}\label{thmpositquuadr}
		Let 
		\begin{equation*}
			\alpha^2\leq \frac{	D-1}{n-1}
		\end{equation*}
		and functions $f$ such that 
		\begin{itemize}
			\item[(i)] $f\in C^3(\overline{\R^d_*}\setminus\{0\})$
			\item[(ii)] $L_S f \in C^1(\overline{\R^d_*}\setminus\{0\})$.
		\end{itemize}
		Then it holds
		\begin{equation*}
			\int_{\mathbb{S}^{d-1}_*} \left(\Gamma_2^S(f)-\alpha^2(n-1)\Gamma_S(f)-\frac{1}{n}(L_Sf)^2\right)d\mu_{\mathbb{S}^{d-1}_*} \geq 0 
		\end{equation*}
	\end{theorem}
	\begin{proof}
		The proof is the same as Theorem \ref{thmintcd}'s. The only difference is that we employ 
		\begin{equation}\label{eq5.84}
			\int_{\mathbb{S}^{d-1}_*} \Gamma^\theta_2(f)d\mu_{\mathbb{S}^{d-1}_*}\geq (D-1)\int_{\mathbb{S}^{d-1}_*} \Gamma^\theta(f)d\mu_{\mathbb{S}^{d-1}_*}
		\end{equation}
		in place of \eqref{eq5.60}. Let us prove \eqref{eq5.84}. Integrating the $CD(D-2,D-1)$ condition on $\mathbb{S}^{d-1}_*$ and using \eqref{eqsL=G2} we end up with
		\begin{align*}
			\int_{\mathbb{S}^{d-1}_*} \Gamma^\theta_2(f)d\mu_{\mathbb{S}^{d-1}_*}&\geq (D-2)\int_{\mathbb{S}^{d-1}_*} \Gamma^\theta (f)  d\mu_{\mathbb{S}^{d-1}_*}+\frac{1}{D-1}\int_{\mathbb{S}^{d-1}_*}(L_\theta f)^2 d\mu_{\mathbb{S}^{d-1}_*}\\
			&=(D-2)\int _{\mathbb{S}^{d-1}_*}\Gamma^\theta (f)  d\mu_{\mathbb{S}^{d-1}_*}+\frac{1}{D-1}\int_{\mathbb{S}^{d-1}_*} \Gamma^\theta_2(f) d\mu_{\mathbb{S}^{d-1}_*}
		\end{align*}
		and so \eqref{eq5.84} is proved.
	\end{proof}
	\section{Proof of Theorem \ref{thm1.1}}\label{sec3}
	This section is devoted to proving Theorem \ref{thm1.1}. Let us make some remarks at first. Call $Z=\mu_S(S)$ and recall that due to $n$-conformal invariance the optimal constant $C_{opt}$ in \eqref{eqsobE} is the same as in \eqref{eqsobS}. Thanks to Proposition \ref{prop1h0} we can test \eqref{eqsobS} with $F=\mathbf{1}$ and deduce that necessarily
	\begin{equation}\label{eqboundCopt}
		C_{opt}\geq \frac{4}{\alpha^2 n (n-2)Z^{\frac{2}{n}}};
	\end{equation}
	this is one of the advantages of working under a finite measure.
	Therefore, if we prove that \eqref{eqsobS} is valid for a constant $C$ equal to  $\frac{4}{\alpha^2 n (n-2)Z^{\frac{2}{n}}}$, \eqref{eqboundCopt} tells us that the optimal constant, by definition, is then equal to $\frac{4}{\alpha^2 n (n-2)Z^{\frac{2}{n}}}=C_{opt}$.\\
	Define now the probability measure $\overline{\mu_S}:=\frac{\mu_S}{Z}$. 
	Observe that, by a classical argument, the compact embedding \eqref{eqrellich} and the fact that $\mathbf{1}\in H^1_0(S)$ imply the validity of a Poincar\'e's inequality on $S$
	\begin{equation}\label{eqpoincare}
		\int_S f^2 d\overline{\mu_S} -\left(\int_S f  d\overline{\mu_S}\right)^2\leq C_P\int _S\Gamma_S(f)d \overline{\mu_S}\qquad\forall f\in \mathcal{A}_0.
	\end{equation}
	It can be proved (see \cite{BakryGentilLedoux}) that Poincar\'e, Sobolev inequalities and $\mathbf{1}\in H^1_0(S)$ imply the validity of a tight Sobolev inequality, namely
	\begin{equation}\label{eqtightsob}
		\left(\int_S |f|^pd\overline{\mu_S}\right)^{2/p}\leq A_p\int_S \Gamma_S (f)d\overline{\mu_S}+\int_S |f|^2 d\overline{\mu_S}
	\end{equation}
	Thus, if we prove that \eqref{eqtightsob} is valid with $A_p=\frac{4}{\alpha^2 n(n-2)}$ we deduce that \eqref{eqsobS} is valid with $C=\frac{4}{\alpha^2 n(n-2)Z^{\frac{2}{n}}}$, so the thesis. Let us start the proof of \eqref{eqtightsob}.\\
	Fix $q<p$ and consider a constant $A_q\in \R$ to be determined. Consider the minimization problem 
	\[
	I(A_q)=\text{inf}\left\{A_q\int_S \Gamma_S(f)d\overline{\mu_S}+\int_S f^2 d\overline{\mu_S}\,:\,f\in H^1_0(S), \norm{f}_{L^q(S,d\overline{\mu_S})}=1\right\}
	\]
	If we take $\textbf{1}$ as a test function, we see that $I(A_q)\leq 1$. As a consequence, 
	\begin{align}\label{eqqsob}
		\left(\int_S |f|^q d\overline{\mu_S}\right)^{2/q}\leq A_q\int_S \Gamma_S(f)d\overline{\mu_S}+\int_S f^2 d\overline{\mu_S}
	\end{align}
	holds if and only if $I(A_q)=1$. Due to \eqref{eqrellich} and Banach-Alaoglu, there exists a minimizer $v$. Without loss of generality, since $|v|$ is also a minimizer, we assume $v\geq 0$ and, prior multiplying it by a constant, we have that $v$ is a weak solution to 
	\begin{equation}\label{eqeq}
		-A_qL_S v+v=v^{q-1}\text{ in }\R^d_*,
	\end{equation}
	namely that
	\begin{equation}\label{eqeqfaible}
		\int_S \left(A_q \Gamma_S(v,\psi)+v\psi\right)\, d\overline{\mu_S}= \int_S v^{q-1} \psi\, d\overline{\mu_S}\qquad\forall \psi \in H^1_0(S).
	\end{equation}
	Observe that local regularity results (see  \cite[Theorem 7.2, Corollary 7.3]{CoraFiorPagliaVita}) tell that $v\in C^3(\overline{\R^d_*}\setminus\{0\})$ and so that \eqref{eqeq} is valid pointwise in $\overline{\R^d_*}\setminus\{0\}$.
	Thanks to the validity of a Sobolev inequality, we can apply Moser's iteration scheme to deduce a global bound on the solution, i.e.
	\begin{equation}\label{eqmoser}
		v\leq c_2.
	\end{equation}
	Also, in paragraph \ref{applowerboundonthesol} we prove a lower bound on the solution, i.e.
	\begin{equation}\label{eqlbsol}
		0<c_1\leq v.
	\end{equation}
	To conclude, in the following lines we show that
	\begin{equation}\label{eqbcarre}
		\Gamma_S(v)\in L^2(S).
	\end{equation}
	In the variables $(y,\theta)$, \eqref{eqeq} is rewritten as 
	\begin{equation}\label{eq3.10}
		(1-y^2)\partial_{yy} v-ny \partial _y v+\frac{1}{\alpha^2(1-y^2)}L_\theta v=\frac{-v^{q-1}+v}{A_q\alpha^2}.
	\end{equation}
	We now want to get rid of $\partial _y v$. To do so, we define $w:=\frac{n}{2}\log(1-y^2)$ and set $h:=e^{\frac{w}{2}}v$. Then \eqref{eq3.10}  becomes 
	\begin{equation*}
		(1-y^2)\partial_{yy} h -\frac{n(n-4)}{4(1-y^2)}h+\frac{L_\theta h}{\alpha^2(1-y^2)}=h\left(\frac{1}{A_q\alpha^2}-\frac{n(n-2)}{4}\right)-\frac{e^\frac{w}{2}}{A_q \alpha^2}v^{q-1}=:R
	\end{equation*}
	and so
	\begin{equation}\label{eq3.12}
		\partial_{yy} h -\frac{n(n-4)}{4(1-y^2)^2}h+\frac{L_\theta h}{\alpha^2(1-y^2)^2}=\frac{R}{(1-y^2)}
	\end{equation}
	with $|R|\leq c(1-y^2)^\frac{n}{4}$ thanks to \eqref{eqmoser}. Consider now $y\in (-1,0]$ and perform a change of variables $t=1+y\in (0,1]$. \eqref{eq3.12} becomes
	\begin{align*}
		\partial_{tt}h-\frac{n(n-4)}{4t^2(2+t)^2}h+\frac{L_\theta h}{\alpha^2 t^2(2+t)^2}=\frac{R}{t(2+t)}=:\tilde{R}
	\end{align*}
	with $|\tilde{R}|\leq c t^{\frac{n-4}{4}}$. Fix $t\in (0,1)$ and take $z\in (1/2,2)$. Call $\tilde{h}(z,\theta):=h(tz, \theta)$ and observe that $\tilde{h}$ solves 
	\begin{equation}\label{eq3.11}
		\partial_{zz}\tilde{h}-\frac{n(n-4)}{4z^2(2+zt)^2}\tilde{h}+\frac{L_\theta\tilde h}{\alpha^2z^2(2+zt)^2}=\tilde{R}t^2=:S\qquad\text{in }(1/2,2)\times \mathbb{S}^{d-1}_*
	\end{equation}
	where $|S|\leq c t^\frac{n+4}{4}$ and $\tilde{h}\leq c t^\frac{n}{4}$. 
	Take $\psi=\psi(z,\theta)\in  C^\infty((1/2,2)\times \mathbb{S}^{d-1})$, multiply \eqref{eq3.11} by $\psi$, and integrate with respect to the measure $dz d\mu_{\mathbb{S}^{d-1}_*}$. Integrating by parts we obtain
	\begin{equation*}
		\int_{1/2}^2 \int_{{\mathbb{S}^{d-1}_*}} \left(\partial _z \tilde h\partial_z \psi+\frac{\Gamma^\theta(\tilde h, \psi)}{\alpha^2z^2(2+zt)^2}\right)dz d\mu_{\mathbb{S}^{d-1}_*}=\int_{1/2}^2 \int_{{\mathbb{S}^{d-1}_*}} \left(S  +\frac{n(n-4)}{4z^2(2+zt)^2}\tilde{h}\right)\psi dz d\mu_{\mathbb{S}^{d-1}_*}.
	\end{equation*}
	We can then use $C^{1,\alpha}$ estimates of \cite[Theorem 1.2]{CoraFiorPagliaVita} to deduce, due to compactness of the sphere, that 
	\begin{equation*}
		|\partial_z \tilde h|+|\nabla_\theta\tilde h|\leq c \left(\norm{\tilde h}_{L^\infty((1/2,2)\times \mathbb{S}^{d-1}_*)}+\norm{S}_{L^\infty((1/2,2)\times \mathbb{S}^{d-1}_*)}\right)\leq c t^\frac{n}{4}\quad\text{in }(3/4,3/2)\times \mathbb{S}^{d-1}_*
	\end{equation*} 
	where the constant $c$ is independent of $t$. Coming back to the variable $y$ we have that 
	\begin{equation*}
		|\partial_y h|\leq c (1+y)^\frac{n-4}{4}, \quad \text{for }y\in (-1,0).
	\end{equation*}
	and therefore
	\begin{align*}
		&\Gamma_S(h)\leq c\left((1+y)|\partial _y h|^2+(1+y)^{-1}\Gamma^\theta(h)\right)\leq c (1+y)^\frac{n-2}{2}\\
		&\Gamma_S(v)=\Gamma_S(he^{-\frac{w}{2}})\leq 2 h^2 \Gamma_S(e^{-\frac{w}{2}})+2e^{-w}\Gamma_S(h)\leq c (1+y)^{-1}.
	\end{align*}
	Eventually, 
	\begin{equation*}
		\int_{(-1,0)\times \mathbb{S}^{d-1}_*}\Gamma_S(v)^2d\mu_S\leq c\int_{(-1,0)\times \mathbb{S}^{d-1}_*}(1+y)^{\frac{n}{2}-3}dy<+\infty
	\end{equation*}
	where we use $n>4$. By symmetry, for $y\in (0,1]$, the reasoning is the same, allowing us to prove \eqref{eqbcarre}.\\
	Let us call $\Phi=v^{-\frac{q-2}{2}}$. Since $v$ solves \eqref{eqeq} we have that $\Phi$ solves 
	\begin{equation}\label{eq3.20}
		\Phi L_S \Phi-\frac{\nu}{2}\Gamma_S(\Phi)=-\lambda (\Phi^2-1).\qquad\text{in }\R^d_*
	\end{equation}
	with $\nu=\frac{2q}{q-2}$ and $\lambda =\frac{2}{(\nu-2)A_q}$.
	Thanks to the smoothness of $v$, \eqref{eqmoser}, \eqref{eqlbsol} and \eqref{eqbcarre} we deduce that
	\begin{equation}\label{eq3.21}
		L_S\Phi \in L^2(S)\cap C^3 (\overline{\R^d_*}\setminus\{0\}).
	\end{equation}
	We multiply \eqref{eq3.20} by $L_S(\Phi^{1-\nu})\zeta_k$, with $\zeta_k$ defined in Lemma \ref{lemA.2}, and we perform the same integration by parts as the proof of Theorem 1.7 in \cite{Zug} (which are rigurous thanks to \eqref{eq3.21}). We then get to 
	\begin{align*}
		\int_S \left(\Gamma_2^S(\Phi)-\frac{1}{\nu}(L_S\Phi)^2-\frac{c}{\nu}\Gamma_S(\Phi)\right)\Phi^{1-\nu}\zeta_k d\overline{\mu_S}=o(1)
	\end{align*}
	where $c=\frac{4(\nu-1)}{(\nu-2)A_q}$. Using Theorem \ref{thmg2int} below we can pass to the limit and end up with
	\begin{align*}
		\int_S \left(\Gamma_2^S(\Phi)-\frac{1}{\nu}(L_S\Phi)^2-\frac{c}{\nu}\Gamma_S(\Phi)\right)\Phi^{1-\nu}d\overline{\mu_S}=0.
	\end{align*} 
	Finally, employing Theorem \ref{thmintcd} we have
	\begin{align}\label{eqfinal}
		\left(\frac{1}{n}-\frac{1}{\nu}\right)\int_S(L_S\Phi)^2 \Phi^{1-\nu}d\overline{\mu_S}+\left(\alpha^2(n-1)-\frac{c}{\nu}\right)\int_S \Gamma_S(\Phi)\Phi^{1-\nu}d\overline{\mu_S}\leq 0.
	\end{align}
	If we set $A_q=\frac{4(\nu-1)}{\nu(\nu-2)\alpha^2(n-1)}$ then $\alpha^2(n-1)=\frac{c}{\nu}$. Moreover, $\left(\frac{1}{n}-\frac{1}{\nu}\right)>0$. But since \eqref{eqfinal} holds, we must have $L_S \Phi=0$. Multiplying $L_S \Phi=0$ by $\Phi$ and integrating by parts we get that $\Gamma_S(\Phi)=0$, so $\Phi$ is constant. Eventually, $v=1$, $I(\frac{4(\nu-1)}{\nu(\nu-2)\rho})=1$ and so \eqref{eqqsob} holds with $A_q=\frac{4(\nu-1)}{\nu(\nu-2)\rho}$. Passing to the limit in \eqref{eqqsob} as $q\rightarrow p$, we deduce that \eqref{eqtightsob} is valid for $A_p=\frac{4}{\alpha^2 n(n-2)}$, the thesis. 
	\\
	\
	\\
	It remains to characterize the optimizers.\\
	Suppose that $v$ is a positive optimizer for \eqref{eqtightsob} with $A_p=\frac{4}{\alpha^2 n(n-2)}$, such that $\norm{v}_{L^p(S,d\overline{\mu_S})}=1$. Namely,
	\begin{equation}\label{eqtightsob1}
		\left(\int_S |v|^pd\overline{\mu_S}\right)^{2/p}=\frac{4}{\alpha^2 n(n-2)} \int_S \Gamma_S (v)d\overline{\mu_S}+\int_S |v|^2 d\overline{\mu_S}
	\end{equation}
	Then, 
	\[
	v\in \text{argmin}\left\{\frac{4}{\alpha^2 n(n-2)}\int_S \Gamma_S(f)d\overline{\mu_S}+\int_S f^2 d\overline{\mu_S}\,:\,f\in H^1_0(S), \norm{f}_{L^p(S,d\overline{\mu_S})}=1\right\}
	\]
	and so, up to a multiplicative constant, $v$ is solution to 
	\[
	-\frac{4}{\alpha^2 n(n-2)}L_S v+v=v^{p-1}
	\]
	Defining again $\Phi=v^{-\frac{p-2}{2}}$ and integrating by parts as in the proof of Theorem 1.7 in \cite{Zug}, we deduce that 
	\begin{align}\label{eq3.15}
		\int_S \left(\Gamma_2^S(\Phi)-\frac{1}{n}(L_S\Phi)^2-\alpha^2(n-1)\Gamma_S(\Phi)\right)\Phi^{1-n} d\overline{\mu_S}=0.
	\end{align}
	If $\alpha^2<\frac{D-1}{n-1}$ or $\alpha^2=\frac{D-1}{n-1}<1$, then \eqref{eq3.15} respectively coupled with \eqref{eq5.62} and \eqref{eq5.62b} tells us that $\Gamma^\theta(\Phi)=0$. So, $\Phi$ is a radial function. Multiplying \eqref{eq5.59} by $\Phi^{1-n}$ and integrating in $d\mu_S$, using \eqref{eq3.15} and the radial symmetry of $\Phi$, we deduce that $\Phi$ solves the following equation
	\begin{equation*}
		\partial_{yy}\Phi=0
	\end{equation*}
	So, $\Phi=C+By$ with $C>|B|$ since we have that $v$ is positive and detached from zero.\\
	Let us now verify the inverse implication, namely that every $v$ such that $\Phi=v^{-\frac{p-2}{2}}=C+By$ with $C>|B|$ is an optimizer. Suppose that $\norm{v}_{L^p(S,d\overline{\mu_S})}=1$. This implies that $C^2-B^2=1$. Indeed, calling $N=\frac{C+B}{C-B}>0$,
	\begin{align*}
		\int_S v^p d\mu_S&=\int_S \Phi^{-n}d\mu_S=\int_S \left(\frac{|x|^{2\alpha}(C+B)+C-B}{|x|^{2\alpha}+1}\right)^{-n}\left(\frac{2}{1+|x|^{2\alpha}}\right)^n |x|^{-bp}x^A dx\\
		&=\int_S \left(\frac{2}{|x|^{2\alpha}(C+B)+C-B}\right)^n|x|^{-bp}x^A dx\\&=(C-B)^{-n}\int_S\left(\frac{2}{1+|N^\frac{1}{2\alpha}x|^{2\alpha}}\right)^n|x|^{-bp}x^A dx\\
		&=(C-B)^{-n} N^\frac{+bp-D}{2\alpha}\int_S \left(\frac{2}{1+|w|^{2\alpha}}\right)^n |w|^{-bp}w^A dw = (C-B)^{-n} \left(\frac{C+B}{C-B}\right)^{-\frac{n}{2}}Z\\
		&=(C^2-B^2)^{-\frac{n}{2}} Z.
	\end{align*}
	So, $\norm{v}_{L^p(S,d\overline{\mu_S})}=1$ if and only if $C^2-B^2=1$. Let us now evaluate
	\begin{align}\label{eq3.28}
		\begin{split}
			\Phi L_S \Phi-\frac{n}{2}\Gamma_S(\Phi)&=(-\alpha^2 ny B)(C+y B)-\frac{n}{2}\alpha^2 (1-y^2)B^2=\frac{n\alpha^2}{2}(C^2-B^2+(C+By)^2)\\&=\frac{n\alpha^2}{2}(1-\Phi^2)
		\end{split}
	\end{align}
	which, since $v=\Phi^{-\frac{n-2}{2}}$, implies in turn that 
	\begin{equation}\label{eq3.29}
		-\frac{4}{n(n-2)\alpha^2}L_S v+v=v^{p-1}.
	\end{equation}
	Multiplying \eqref{eq3.29} by $v$, integrating in $d\overline{\mu_S}$ and integrating by parts, (recalling that $\norm{v}_{L^p(S,d\overline{\mu_S})}=1$) we see that $v=(C+By)^{-\frac{n-2}{2}}$ is optimizer for \eqref{eqtightsob1}, the thesis.\\
	To conclude, let us evaluate optimizers $f$ for \eqref{eqmonckn}. From Proposition \ref{sobinvariance}, they are all equal to 
	\begin{align*}
		f=(C+By)^{-\frac{n-2}{2}}\left(\frac{2}{1+|x|^{2\alpha}}\right)^\frac{n-2}{2}=\left(\frac{2}{C-B+|x|^{2\alpha}(C+B)}\right)^\frac{n-2}{2}
	\end{align*}
	and since $C>|B|$, they are all of the form \eqref{eqopt}.
	\subsection{Computation of $C_{opt}$}
	Let us conclude evaluating the constant $Z=\mu_S(\R^d_*)$ in order to entirely explicit $C_{opt}$. Call $u(|x|):=\left(\frac{2}{1+|x|^{2\alpha}}\right)^{n}|x|^{-bp}$. We have then
	\begin{align*}
		Z=\int_{{\mathbb{\R}^{d}_*}} u(|x|)x^A dx=\int_0^{+\infty}u(r)\left(\int_{\partial B_r^*}x^A d\mathcal{H}^{d-1}\right)dr=\left(\int_{\partial B_1^*}x^A d\mathcal{H}^{d-1}\right)\left(\int_0^{+\infty}r^{D-1}u(r) dr \right).
	\end{align*} 
	From \cite[Page 4330-4331]{Ca} (see also \cite{Abram}) we have that
	\begin{equation*}
		\int_{\partial B_1^*}x^A d\mathcal{H}^{d-1}=D\frac{\Gamma(\frac{A_1+1}{2})\cdot...\cdot \Gamma(\frac{A_d+1}{2})}{2^k \Gamma(1+\frac{D}{2})},
	\end{equation*}
	and eventually 
	\begin{align*}
		\int_0^{+\infty}r^{D-1}u(r) dr=\int_0^{+\infty}\left(\frac{r^\alpha+r^{-\alpha}}{2}\right)^{-n}\frac{1}{r}dr=\int_{\R}\cosh(\alpha u)^{-n}du.
	\end{align*}
Moreover, performing some change of varibles,
\begin{align*}
\int_{\R}\cosh(\alpha u)^{-n}du&=\frac{2^n}{\alpha}\int_{\R}\frac{e^{(n-1)u}}{(e^{2u}+1)^n}e^u du=\frac{2^n}{\alpha}\int_0^{+\infty}\frac{w^{n-1}}{(w^2+1)^n}dw=\frac{2^{n-1}}{\alpha}\int_0^{+\infty}\frac{s^{\frac{n}{2}-1}}{(s+1)^n}ds\\&=\frac{2^{n-1}}{\alpha}B\left(\frac{n}{2},\frac{n}{2}\right)=\frac{2^{n-1}}{\alpha}\frac{\left(\Gamma\left(\frac{n}{2}\right)\right)^2}{\Gamma(n)}=\frac{\sqrt{\pi}}{\alpha}\frac{\Gamma\left(\frac{n}{2}\right)}{\Gamma\left(\frac{n+1}{2}\right)}.
\end{align*}
where we used the definition and properties of the Beta function $B(z,w)$ (see \cite[6.2.1]{Abram}) and the Legendre duplication formula\footnote{see https://en.wikipedia.org/wiki/Gamma\_function}.\\
	Putting everything together we obtain
	\begin{align*}
		Z=\frac{ D\sqrt{\pi}}{\alpha}\frac{\Gamma(\frac{A_1+1}{2})\cdot...\cdot \Gamma(\frac{A_d+1}{2})}{2^k \Gamma(1+\frac{D}{2})} \frac{\Gamma\left(\frac{n}{2}\right)}{\Gamma\left(\frac{n+1}{2}\right)},
	\end{align*}
	the thesis.
	\section{Appendix}\label{sec4}
	This appendix is mainly dedicated to regularity results of functions involved in the paper and the validity of the monomial CKN inequality with some costant. 
	Let us begin proving an integrability condition of $\Gamma_2^S(f)$ under some hypothesis on the function $f$ and the integrated curvature dimension condition. 
	\begin{theorem}\label{thmg2int}
		Assume $n>4$ and \eqref{hpalphaint}. If $f\in C^3(\overline{\R^d_*}\setminus\{0\})\cap L^\infty(\R^d_*)$ and $L_S f\in C^1(\overline{\R^d_*}\setminus\{0\})\cap L^2(S)$, $ \Gamma_S(f)\in L^2(S)$, then
		\begin{equation*}
			\int_I\left|\int_{{\mathbb{S}^{d-1}_*}}\Gamma_2^S(f)d\mu_{\mathbb{S}^{d-1}_*} \right|(1-y^2)^{\frac{n}{2}-1}dy<+\infty
		\end{equation*}
	\end{theorem}
	\begin{proof}
		As defined in Proposition \ref{lemA.2}, we approximate $f$ by a sequence of functions $f\zeta_k=f_k\in \mathcal{A}_0$ such that $f\zeta_k=f_k\rightarrow f$ in $D(L_S)$ and $f=f_k$ in $(-1+2/k,1-2/k)\times \mathbb{S}^{d-1}_*$. Let us remark that, since $\zeta_k$ is radial, $f_k$ satisfies hypotheses of Theorem \ref{corgam2}. If we set $\gamma_2(h)(y):=\int_{{\mathbb{S}^{d-1}_*}}\Gamma_2^S(h)d\mu_{\mathbb{S}^{d-1}_*}$ for $h\in C^3(\overline{\R^d_*}\setminus\{0\}),\,L_S h \in C^1(\overline{\R^d_*}\setminus\{0\})$, then Theorem \ref{thmpositquuadr} tells us that $\gamma_2$ is a positive quadratic form. Therefore, Cauchy-Schwarz inequality holds
		\begin{equation}\label{eq3.16}
			|\gamma_2(f_k)^{1/2}-\gamma_2(f_l)^{1/2}|^2\leq \gamma_2(f_k-f_l).
		\end{equation}
		Moreover, integrating \eqref{eq3.16} with respect to $(1-y^2)^{\frac{n}{2}-1}dy$ and using Lemma \ref{corgam2}, we deduce that $\gamma_2(f_k)$ is a Cauchy sequence in $L^2(I,(1-y^2)^{\frac{n}{2}-1}dy)$. The thesis follows.
	\end{proof}
	The following part is devoted to proving \eqref{eqmonckn} with a costant which need not be optimal. It is done employing \eqref{eqcabros} (proved in \cite{Ca}) and showing that it is $n$-conformal to a simpler model. The conclusion follows using Proposition \ref{sobinvariance} and an interpolation inequality.
	\begin{proposition}
		If $0\leq b-a\leq 1$ and $a\neq \frac{D-2}{2}$, \eqref{eqmonckn} holds true. 
	\end{proposition}
	\begin{proof}
		First, let us observe that $E$ is $n$-conformal to the cylinder $C:=\R\times \mathbb{S}^{d-1}_*$. 		
		Indeed, calling $|x|=r=e^s$,
		\begin{align}\label{eq4.1}
			\begin{split}
				&g_E = r^{2(\alpha-1)}g_{\R^d}= r^{2(\alpha-1)}( dr^2+r^2 g_{\mathbb{S}^{d-1}})=e^{2\alpha s}(ds^2+g_{\mathbb{S}^{d-1}})= e^{2\alpha s} g_C\\
				&d\mu_E=e^{n\alpha  s}dsd\mu_{\mathbb{S}^{d-1}_*}=e^{n\alpha s} d\mu_C
			\end{split}
		\end{align}
		Then, \eqref{eqmonckn} is equivalent to 
		\begin{equation}\label{sobcilinder}
			\left(\int_C |f|^p d\mu_C\right)^{2/p}\leq c\left(\int_C \left(\left(\frac{\partial f}{\partial s}\right)^2+|\nabla _\theta f|^2\right)d\mu_C+\frac{\alpha^2(n-2)^2}{4}\int_C f^2d\mu_C\right).
		\end{equation}
		Indeed, calling $u=e^{\alpha s\frac{n-2}{2}}$ and $L_C=\frac{\partial^2}{\partial s^2}+L_\theta$, we have
		\[
		-L_C u+\frac{\alpha^2(n-2)^2}{4} u=0
		\]
		and we can apply Proposition \ref{sobinvariance}.\\
		Now, call $g=g_{\mathbb{\R}^d}$, $d\mu=x^A dx$.
		From \cite{Ca} we know that 
		\begin{equation}\label{eqrosmon}
			\left(\int_{\R^d_*} |f|^q d\mu\right)^{2/q}\leq c \left(\int_{\R^d_*}|\nabla f|^2 d\mu\right)\qquad\forall f\in C_c^\infty(\R^d)\qquad q=\frac{2D}{D-2}.
		\end{equation}
		Fix $f\in C^\infty_c(\R^d\setminus\{0\})$. Call $u=\psi^{\frac{D-2}{2}}$ and $f=uF$, so $F\in C^\infty_c(\R^d\setminus\{0\})$. Observe that the model $(\R^d_*, g, d\mu)$ is trivially equal to $E$ if $a=b=0$ and so it is $D$-conformal to $C$ with $\alpha=1$, by \eqref{eq4.1}. Therefore, we can perform the computations of Proposition \ref{sobinvariance} and use integration by parts on $\R^d_*$ to get
		\begin{equation}\label{eq6.1}
			\int_C \left(\frac{\partial F}{\partial s}\right)^2+|\nabla_\theta F|^2d \tilde \mu+\frac{(D-2)^2}{4}\int_C F^2 d\tilde{\mu}=\int_{\R^d_*} |\nabla f|^2 d\mu
		\end{equation}
		since $-(\Delta u-\nabla \log x^A\nabla u)u^{1-p}=\frac{(D-2)^2}{4}$.\\
		Observe that 
		\begin{align}\label{eq6.3}
			\int_C F^2 d\tilde{\mu}=\int_{\R^d_*} f^2 \psi^2 d\mu
		\end{align}
		\begin{align}\label{eq6.4}
			\int_C |F|^pd\tilde \mu =\int_{\R^d_*} f^p u^{-p}\psi^Dd\mu = \int_{\R^d_*} |f|^p \psi ^r d\mu \qquad\text{with } r=D+\frac{p}{2}(2-D).
		\end{align}
		Now, we need to adjust the Sobolev exponents in order to prove the thesis, since $p< q$. For this reason we interpolate $p$ between $2$ and $q$. Let $\theta\in (0,1)$ such that 
		\[
		\frac{1}{p}=\frac{\theta}{2}+\frac{1-\theta}{q},
		\]
		therefore,
		\begin{align}\label{eq6.90}
			\int_{\R^d_*} |f|^p \psi ^r d\mu \leq \left(\int_{\R^d_*}(|f|^{\theta p} \psi ^r)^\frac{2}{\theta p}d\mu\right)^\frac{\theta p}{2}\left(\int_{\R^d_*} |f|^q d\mu\right)^\frac{p(1-\theta)}{q}.
		\end{align}
		A computation shows that  $\frac{2r}{\theta p}=2$.
		Indeed, 
		\begin{align*}
			\frac{2r}{\theta p}=\frac{p(2-D)+2D}{\theta p},\qquad \theta p=\left(\frac{q-p}{q} \right)D=\left(\frac{\frac{2D}{D-2}-p}{\frac{2D}{D-2}}\right) D
		\end{align*}
		and so
		\begin{align*}
			\frac{2r}{\theta p}=2.
		\end{align*}
		Thus \eqref{eq6.90} becomes
		\begin{align}\label{eq6.2}
			\int_{\R^d_*} |f|^p \psi ^r d\mu \leq \left(\int_{\R^d_*}f^2 \psi ^2d\mu\right)^\frac{\theta p}{2}\left(\int_{\R^d_*} |f|^q d\mu\right)^\frac{p(1-\theta)}{q}
		\end{align}
		Eventually, bringing everything together, 
		\begin{align}\label{eq6.95}
			\begin{split}
				&\int_C |F|^p d\tilde \mu \overset{\eqref{eq6.4}, \eqref{eq6.2}}{\leq } \left(\int_{\R^d_*}f^2 \psi ^2d\mu\right)^\frac{\theta p}{2}\left(\int_{\R^d_*} |f|^q d\mu\right)^\frac{p(1-\theta)}{q}\overset{\eqref{eq6.3}}{=}\left(\int_C F^2 d\tilde{\mu}\right)^\frac{\theta p}{2} \left(\int_{\R^d_*} |f|^q d\mu\right)^\frac{p(1-\theta)}{q}\\
				& \overset{\eqref{eqrosmon}}{\leq }c\left(\int_C F^2 d\tilde{\mu}\right)^\frac{\theta p}{2} \left(\int_{\R^d_*}|\nabla f|^2 d\mu\right)^\frac{p(1-\theta)}{2}\overset{\eqref{eq6.1}}{\leq }c \left(\int_C \left(\frac{\partial F}{\partial s}\right)^2+|\nabla_\theta F|^2d \tilde \mu+\frac{(D-2)^2}{4}\int_C F^2 d\tilde{\mu}\right)^\frac{p}{2}.
			\end{split}
		\end{align}
		We just proved that \eqref{sobcilinder} holds for all functions in $C^\infty_c(\R^d\setminus\{0\})$.\\
		To sum up, we proved the following sequence of implications.
		\[
		\text{monomial Sobolev inequality }\,\eqref{eqcabros}\Rightarrow\text{Sobolev on the cylinder }\,\eqref{sobcilinder}\Rightarrow\text{monomial CKN inequality }\,\eqref{eqmonckn}
		\]
	\end{proof}
	\begin{remark}
		Observe that \eqref{eqrosmon} holds for functions which need not vanish around the origin. In order to deduce \eqref{eq6.1} and apply integration by parts, we must work with functions which vanish at the origin, being the conformal factor $\psi$ degenerate in zero.
	\end{remark}
	Let us now comment on the range of parameters of $a,b$ for which the inequality \eqref{eqmonckn} holds in the previous proof.
	When we interpolate, we say that $2\leq p\leq q$. This is equivalent to 
	\[
	0\leq b-a\leq 1.
	\]
	In addition, where does the hypothesis $a\neq a_c:= \frac{D-2}{2}$ come into play? This condition indeed is necessary, being the standard CKN not valid for this value of $a$. Since we have 
	\begin{equation*}
		2a=(D-2)-\alpha(n-2)
	\end{equation*}
	then $a=a_c$ if and only if $\alpha =0$. 
	Actually, in our previous proof, $\alpha$ cannot be equal to zero. Indeed, if it were, we would have that \eqref{eqmonckn} woud be equivalent to  \eqref{sobcilinder} with $\alpha=0$, namely,
	\begin{align}\label{eq4.22}
		\left(\int_C |f|^p d\mu_C\right)^{2/p}\leq c\int_C \left(\left(\frac{\partial f}{\partial s}\right)^2+|\nabla _\theta f|^2\right)d\mu_C.
	\end{align}
	However, in \eqref{eq6.95} we proved that \eqref{eqcabros} implies the validity of
	\begin{align*}
		\left(\int_C |f|^p d\mu_C\right)^{2/p}\leq c\left(\int_C \left(\frac{\partial f}{\partial s}\right)^2+|\nabla _\theta f|^2d\mu_C+\frac{(D-2)^2}{4}\int_C f^2 d\mu_C\right)
	\end{align*}
	making it then impossible to conclude  since in the RHS of \eqref{eq4.22} the $L^2$ norm of the function is missing. \\ \ \\
	In the following subsection we prove the lower bound on the solution of \eqref{eqeq} via Markov semigroups. Unlike the standard CKN inequality, we could not prove the essential self adjointness of $L_S$. Therefore, we had to adapt some proofs of \cite{BakryGentilLedoux}.
	\subsection{Lower bound on the solution}\label{applowerboundonthesol}
	This section is devoted to the proof of \eqref{eqlbsol}. It is done in multiple steps using results belonging to the theory of Markov semigroups. Define the semigroup $P_t f$ for all $f\in L^2(S)$ as 
	\[
	P_t f:=\sum_{k=0}^{+\infty} e^{-\lambda_k t}f_k \phi_k \in L^2(S)
	\]
	where $\phi_k\in H^1_0(S)$ is an orthonormal base of eigenfunctions in $L^2(S)$ of the operator $-L_S$, and $\lambda_k$ are the respective eigenvalues. It is known that $\lambda_0=0$, $\lambda_k>0$ for $k\geq 1$ and $\lambda_k\rightarrow +\infty$. Their existence is guaranteed by a classical argument of compactness of the map $T_1:=(-L_S+\mathbb{I})^{{-1}}$ (see \eqref{eqres} for the definition) and the spectral theorem (see \cite{Evans, Brezis}). 
	We have that $\norm {P_t f}_{2}\leq \norm {f}_2$ ($P_t$ is contracting), so that $P_t$ is a linear and continuous operator from $L^2(S)$ to $L^2(S)$. Moreover, $P_t \phi_k= e^{-\lambda _k t} \phi_k$ (by orthogonality of the base). This is enough to prove the semigroup property $P_{t+t'} f=P_t P_{t'} f$.\\
	It is also valid that $P_t f \in H^1_0(S)$ (by the completeness of this space), and that 
	\[
	\mathcal{E}(P_t f):=\int_S \Gamma (P_t f)d\mu_S =\sum_{k=0}^{+\infty} \lambda_k e^{-2\lambda_k t}f_k^2.
	\]
	Indeed, if we call $A_n:=\sum_{k=0}^n e^{-\lambda_k t} f_k \phi_k$ we have that by definition $A_n\xrightarrow{n\rightarrow +\infty} P_t f$ in $L^2(S)$. Also, being $A_n$ a finite sum, $A_n \in H^1_0(S)$ and by orthogonality (and Parseval's identity) $\int \Gamma (A_m-A_n)=\sum_{k=n+1}^m \lambda_k e^{-2\lambda_k t} f_k^2\xrightarrow{n,m \rightarrow +\infty}0$. So, there exists the square of the weak derivative $\Gamma (P_t f)$. Moreover, since $\sum_{k=0}^n \lambda_k e^{-2\lambda _k t} f_k^2=\int \Gamma (A_n)\xrightarrow{n\rightarrow +\infty} \int \Gamma (P_t f)$ (by continuity of the $L^2$ norm) we have
	\[
	\mathcal{E}(P_t f)=\sum_{k=0}^{+\infty} \lambda_k e^{-2\lambda _k t }f_k^2.
	\] 
	\subsubsection{Notable properties of $P_t$}\ \\
	In this part we prove the boundedness of $P_t$ from $L^1(S)$ to $L^\infty(S)$ and some other useful properties.
	Let us define
	\[
	\Lambda (t):=\int_S (P_t f)^2 d\mu_S=\sum_{k=0}^{+\infty}e^{-2\lambda_k t} f_k^2.
	\]
	$\Lambda (t)$ is differentiable for all $t>0$ by Weierstrass theorem and 
	\begin{equation}\label{lamprim}
		\Lambda '(t)=-2\mathcal{E}(P_t f).
	\end{equation}
	\eqref{lamprim} is enough to apply Theorem 4.2.5 in \cite{BakryGentilLedoux} which affirms that 
	\begin{equation}\label{eqvar}
		\text{Var}(P_t f)\leq e^{-tD}\text{Var}(f)\qquad\forall f\in L^2(S)
	\end{equation}
	where $\text{Var}(h):=\int_S h^2 d\mu_S-\left(\int_S h d\mu_S\right)^2$.\\
	Observe now that Sobolev inequality implies Nash's inequality. Indeed, it is sufficient to interpolate between $1$ and $p=\frac{2n}{n-2}$. We have that 
	\[
	\frac{1}{2}=\theta \frac{n-2}{2n}+(1-\theta)\qquad\text{with }\theta=\frac{n}{n+2}
	\]
	and so 
	\[
	\norm{f}_2\leq \norm{f}_p^\theta \norm{f}_1^{1-\theta}=\norm{f}_p^\frac{n}{n+2}\norm{f}_1^\frac{2}{n+2}.
	\]
	Applying now Sobolev's inequality we get the following Nash's inequality
	\begin{equation}\label{nash}
		\norm{f}_2^{n+2}\leq \left[A \norm{f}_2^2+C\mathcal{E}(f)\right]^{\frac{n}{2}}\norm{f}_1^2.
	\end{equation}
	Using the equality \eqref{lamprim} we can rewrite \eqref{nash} applied to $P_t f$ for $f>0$ and $\int_S f d\mu_S =\int_S P_t f d \mu_S=1$ as 
	\[
	\Lambda (t)\leq \left[A\Lambda (t)-\frac{C}{2}\Lambda'(t)\right]^\frac{n}{n+2}.
	\]
	This implies, from Gr\"onwall's lemma, that the function $e^{\lambda t}(1-A\Lambda^{-r}(t))$ (with $r=\frac{n}{2}$ and $\lambda=\frac{2Ar}{C}$) is decreasing in $t\geq 0$. As a consequence
	\[
	\Lambda(t)\leq \left(\frac{A}{1-e^{-\lambda t}}\right)^{n/2}
	\]
	which in turns implies that $P_t$ is bounded from $L^1(S)$ to $L^2(S)$. From duality we have that $P_t^*$ is bounded from $L^2$ to $L^\infty$. Also, $P_t^*=P_t$ since, for each $w,v\in L^2(S) $, we deduce
	\[
	(P_t^* w, v)_{L^2}=(w,P_t v)_{L^2}=\sum_k e^{-\lambda_k t}w_k v_k=(P_t w, v)_{L^2}.
	\]
	By composition, and using the property of semigroups, we get that $P_{2t}$ is bounded from $L^1(S)$ to $L^\infty(S)$. 
	Define $P_t^0 f:=P_t f-\int_S f d\mu_S $. We want to prove that it is bounded from $L^1(S)$ to $L^\infty(S)$ with norm less equal than $C^2e^{-Dt}$.
	Indeed, from the bound of $P_t$ from $L^1(S)$ to $L^\infty(S)$, we deduce also that $\norm{P_1^0}_{1,\infty}\leq C$, and therefore $\norm{P_1^0}_{2,\infty}\leq C$, $\norm{P_1^0}_{1,2}\leq C$. Moreover, \eqref{eqvar} tells us that $\norm{P_t^0 f}_2\leq e^{-tD}\norm{f}_2$. Using now the semigroup property we get $P_{2+t}^0=P_1^0\circ P^0_t \circ P^0_1$ and so that $\norm{P^0_{2+t}}_{1,\infty}\leq C^2 e^{-tD}$.\\
	Fix $f\in L^2(S)$, $f>0$. We can also see $P_t$ as an Hilbert-Schmidt operator. Call its kernel $p_t(x,y)$. The bound on the $1,\infty$ norm of $P^0_{2+t}$ tells us that 
	\begin{align*}
		|P^0_{2+t} f|=\left|\int_S [p_{t+2}(x,y)-1] f(y)d\mu_S(y)\right|\leq C ^2e^{-Dt} \int_S f(y) d\mu_S(y)\qquad \text{for almost every }x\in \R^d_*
	\end{align*}
	Which yields in particular that, for each $t>T$, with $T$ sufficiently big,
	\begin{align}\label{eqlbou}
		\begin{split}
			c\int_S f(y)d\mu_S(y)&\leq(1-C^2e^{-Dt})\int_S f(y)d\mu_S(y) \\ &\leq\int_S p_{t+2}(x,y)f(y)d\mu_S(y)=P_{2+t}f\qquad \text{for almost every }x\in \R^d_*.
		\end{split}
	\end{align}
	with $c$ independent of $t>T$.\\
	\subsubsection{The operator $T_\lambda$ and the conclusion}\ \\
	Fix now $\lambda >0$ and consider the operator 
	\begin{equation}\label{eqres}
		T_\lambda:=(-L+\lambda \mathbb{I})^{-1}
	\end{equation} defined from $L^2(S)$ to $H^1_0(S)$ such that for each $f\in L^2(S)$, $T_\lambda f=u$ is the only function in $H^1_0(S)$ for which
	\[
	\int_S \Gamma(u,\psi)d \mu_S+\lambda \int_S u \psi d\mu_S=\int_S f \psi d\mu_S.\qquad\forall \psi\in H^1_0(S).
	\] 
	Consider now the operator $G_\lambda:=\int_{0}^{\infty} P_t\, e^{-\lambda t}dt$. 
	For each $f\in L^2(S)$ we have 
	\begin{equation}\label{normG}
		\norm{\int_{0}^{\infty} P_t f e^{-\lambda t}dt}_2\leq \int_{0}^{\infty}\norm{P_t f}_2e^{-\lambda t}dt\leq \int_{0}^{\infty}\norm{f}_2e^{-\lambda t}dt=\frac{1}{\lambda}\norm{f}_2
	\end{equation}
	so that $G_\lambda$ is bounded from $L^2(S)$ to $L^2(S)$. Also, $T_\lambda$ is continuous from $L^2(S)$ to $L^2(S)$. We now wish to prove that
	\begin{equation}\label{eqG=T} 
		G_\lambda=T_\lambda
	\end{equation}	
	By continuity and linearity, it is enough to prove that $T_\lambda \phi_k =G_\lambda  \phi _k$ for each $k$ i.e. that 
	\begin{equation*}\label{eq1}
		\int_S \Gamma(\int_{0}^{\infty} P_t \phi_k e^{-\lambda t}dt,\psi)d \mu_S+\lambda \int_S \left(\int_{0}^{\infty} P_t \phi_k e^{-\lambda t}dt\right) \psi d\mu_S=\int_S \phi_k \psi d\mu_S. \qquad\forall \psi\in H^1_0(S).
	\end{equation*}
	However, $P_t \phi_k=e^{-\lambda_k t}\phi_k$, so 
	\[
	\int_{0}^{\infty} P_t \phi_k e^{-\lambda t}dt=\phi_k \frac{1}{\lambda_k+\lambda}.
	\]
	Being $\phi_k$ eigenfunction we deduce that 
	\[
	\int_S \Gamma(\int_{0}^{\infty} P_t \phi_k e^{-\lambda t}dt,\psi)d \mu_S=\frac{\lambda_k }{\lambda+\lambda_k}\int_S \phi_k \psi d \mu_S
	\]
	so \eqref{eqG=T} is proved.
	In particular, \eqref{eqG=T} and \eqref{normG} allow us to deduce that
	\begin{equation}\label{eqnormT}
		\norm{T_\lambda}_{2,2}\leq \frac{1}{\lambda}
	\end{equation} We have now all the ingredients to prove \eqref{eqlbsol}. Since $v$ is solution to \eqref{eqeq}, we have that $v=(-L+\frac{1}{A_q}\mathbb{I})^{-1}( \frac{1}{A_q}v^{q-1})$. Define $\lambda:= \frac{1}{A_q}$. Applying \eqref{eqlbou} we know that there exists a $T>0$ such that
	\begin{align}\label{eq4.33}
		v=(-L+\lambda \mathbb{I})^{-1}(\frac{1}{A_q}v^{q-1})=\frac{1}{A_q}\int_{0}^{\infty}P_t v^{q-1}e^{-\lambda t}dt \overset{\eqref{eqlbou}}{\geq}\frac{c}{A_q} \int_S v^{q-1}(y)d\mu_S(y)\int_{T}^{\infty}e^{-\lambda t}dt=c_1
	\end{align}
	finding \eqref{eqlbsol}, the thesis.
	Actually, \eqref{eq4.33} relies also on the positivity preserving property (to deduce that $\int_{0}^{T}P_t v^{q-1}e^{-\lambda t}dt \geq 0$). Namely that if $f\geq 0$ then $P_t f\geq 0$. 
	Let us prove it briefly in the following paragraph.\\
	\subsubsection{Positivity preserving}\ \\
	First, let us prove the following identity
	\begin{equation}\label{eqidop}
		P_t f=\lim_{\lambda \rightarrow +\infty} e^{-\lambda t}\sum_{k\in \N}\left(\frac{t^k}{k!}(\lambda ^2 T_\lambda )^k f \right) =:A_t f\qquad f\in L^2(S)
	\end{equation}
	where $T_\lambda$ is defined as in \eqref{eqres} and the limit has to be intended as a limit in $L^2(S)$, for each $f$ fixed.
	By definition $T_\lambda \phi_j=\frac{1}{\lambda +\lambda_j}\phi_j$ and so 
	\[
	\sum_{k\in \N}\frac{t^k}{k!}\lambda ^{2k} T_\lambda^k \phi_j =e^{\frac{\lambda^2 t}{\lambda+\lambda_j}} \phi_j.
	\]
	Let us call $M_{\lambda,t}:=\sum_{k\in \N}\frac{t^k}{k!}\lambda ^{2k}T_\lambda ^k$. \eqref{eqnormT} allows us to deduce that $M_{\lambda,t} $ is a continuous linear operator from $L^2$ to $L^2$ whose norm is 
	\begin{equation*}
		\norm{M_{\lambda,t}}_{2,2}\leq e^{\lambda t}.
	\end{equation*}
	Take a function $f=\sum_{j\in \N} c_j \phi_j\,\, \in L^2(S)$. By continuity and linearity of $M_{\lambda,t}$ we deduce that \begin{equation*}
		M_{\lambda,t} f=\sum_j^\infty c_j M_{\lambda,t} \phi_j=\sum_j^\infty c_j e^{\frac{\lambda^2 t}{\lambda +\lambda_j}}\phi_j \qquad\text{in }L^2(S).
	\end{equation*}
	Eventually, by the orthogonality of the base $\phi_j$,
	\begin{align*}
		\norm{e^{-\lambda t}M_{\lambda,t} f-P_t f}^2_2=\norm{\sum_j^\infty c_j e^{\frac{-\lambda \lambda_j t}{\lambda +\lambda_j}}\phi_j-\sum_j^\infty c_j e^{-\lambda_j t}\phi_j}^2_2=\sum_j^\infty c_j^2 \left(e^{\frac{-\lambda \lambda_j t}{\lambda +\lambda_j}}-e^{-\lambda_j t}\right)^2\xrightarrow{\lambda \rightarrow \infty}0
	\end{align*}
	Proving \eqref{eqidop}.
	\\
	Let us now make an observation on the stability of the space $H^1_0(S)$. Take $\psi:\R\rightarrow\R$ a smooth function such that $\psi(0)=0$ and $|\psi'|\leq 1$. It is then easy to see that if $u\in H^1_0(S)$ then $\psi(u)\in H^1_0(S)$. Also, 
	\[
	\mathcal{E}(\psi(u))=\int_S |\psi'(u)|^2\Gamma_S(u)d\mu_S
	\]
	and so
	\begin{equation*}
		\mathcal{E}(\psi(u))\leq \mathcal{E}(u)\qquad\forall u\in H^1_0(S)
	\end{equation*}
	which in turn implies
	\begin{equation}\label{eqpopsi}
		\mathcal{E}(u-\psi(u),u)\geq0.
	\end{equation}
	Observe that, by approximation, \eqref{eqpopsi} is also valid for $\psi(r)=\min\{r,1\}$. 
	Our goal is to prove that if $f\leq 1$, then $P_t f \leq 1$.\\
	From the representation \eqref{eqidop} it is sufficient to prove that $h=\lambda T_\lambda (f)=T_\lambda(\lambda f)\leq 1$. By definition of $T_\lambda$ it holds
	\begin{equation*}
		\int _S\Gamma_S(h,\varphi)+\lambda h\varphi\, d\mu_S=\lambda \int_S f\varphi\, d\mu_S \qquad\forall \varphi\in H^1_0(S).
	\end{equation*}
	If we then select $\varphi=h-\psi(h)$ we get
	\begin{align*}
		\lambda \int_S h(h-\psi(h))d\mu_S=&-\mathcal{E}(h-\psi(h),h)+\lambda \int_S f(h-\psi(h))d\mu_S\\
		\overset{\eqref{eqpopsi}}{\leq}&\lambda \int_S f(h-\psi(h))\, d\mu_S.
	\end{align*}
	Eventually, 
	\begin{equation}\label{eq4.19}
		\int_S (f-h)(h-\psi(h))d\mu_S\geq 0
	\end{equation}
	and take $\psi(r)=\min\{r,1\}$. Then \eqref{eq4.19} becomes
	\begin{equation*}
		\int_{\{h\geq 1\}}(f-h)(h-1)d\mu_S\geq 0
	\end{equation*}
	and since $f\leq 1$ we deduce that $\mu_S(\{h\geq 1\})=0$, our thesis.\\
	This is enough to conclude since, if $f\geq 0$, then $w:=-f+1\leq 1$ and by what we just proved $P_t w \leq 1$. But the fact that $P_t$ is linear and $P_t \mathbf{1}=\mathbf{1}$ allows us to conclude that $P_t f\geq 0$.
	\subsection{Curvature dimension condition on $\mathbb{S}^{d-1}_*$}\label{subseccdsphere}\ \\
	In this last section we prove the curvature dimension condition $CD(D-2,D-1)$ on $\mathbb{S}^{d-1}_*$. Recall that $d\mu_{\mathbb{S}^{d-1}_*}:=\theta^A dV_{\mathbb{S}^{d-1}}=e^{-W_\theta} dV_{\mathbb{S}^{d-1}}$. Therefore, using the stereographic projection from the north pole of the sphere $\mathbb{S}^{d-1}$ and calling $\varphi=\frac{1+|x|^2}{2}$, with $x\in \R^{d-1}$, we have 
	\begin{align*}
		&\mathbb{S}^{d-1}\ni(\theta_1,...,\theta_{d-1},\theta_{d})=\left(\frac{x_1}{\varphi},..., \frac{x_{d-1}}{\varphi},\frac{|x|^2-1}{|x|^2+1}\right),\\
		&W_\theta=|A|\log \varphi-\sum_i A_i\log x_i.
	\end{align*}
	Let us now evaluate the following quantities. From now on $f$ is a smooth function on $\mathbb{S}^{d-1}$.
	\begin{align*}
		&\Gamma^\theta(W_\theta,f)=\varphi^2\nabla W_\theta \cdot \nabla f=|A|\varphi x\cdot \nabla f-\varphi^2\sum_i \frac{A_i}{x_i}\frac{\partial f}{\partial x_i},\\
		&\Gamma^\theta(\Gamma^\theta(W_\theta,f),f)=|A|\varphi^2(x\cdot \nabla f)^2+|A|\varphi^3 \nabla (x\cdot \nabla f)\cdot \nabla f\\
		&-2\varphi^3\left(\sum_i \frac{A_i}{x_i}\frac{\partial f}{\partial x_i}\right)x\cdot \nabla f+\varphi^4\sum_i \frac{A_i}{x_i^2}\left(\frac{\partial f}{\partial x_i}\right)^2-\varphi^4\sum_i \frac{A_i}{x_i}\nabla \left(\frac{\partial f}{\partial x_i}\right)\cdot \nabla f\\
		&=|A|\varphi^2(x\cdot \nabla f)^2+|A|\varphi^3(\nabla ^2 f(x,\nabla f)+|\nabla f|^2)\\
		&-2\varphi^3\left(\sum_i \frac{A_i}{x_i}\frac{\partial f}{\partial x_i}\right)x\cdot \nabla f+\varphi^4\sum_i \frac{A_i}{x_i^2}\left(\frac{\partial f}{\partial x_i}\right)^2-\varphi^4\sum_i \frac{A_i}{x_i}\nabla \left(\frac{\partial f}{\partial x_i}\right)\cdot \nabla f,
	\end{align*}
	Denote by $\{e_1,...,e_{d-1}\}$ the canonical basis of $\R^{d-1}$. Then,
	\begin{align*}
		&\Gamma^\theta(\frac{\Gamma^\theta(f)}{2},W_\theta)=\frac{1}{2}\varphi^2\nabla (\varphi^2 |\nabla f|^2)\cdot \left(\frac{|A|x}{\varphi}-\sum_i\frac{A_i e_i}{x_i} \right)\\
		&=|A|\varphi^3|\nabla f|^2(\frac{|x|^2}{\varphi}-1)+\frac{1}{2}\varphi^4 \nabla (|\nabla f|^2)\cdot \left(\frac{|A|x}{\varphi}-\sum_i\frac{A_i e_i}{x_i} \right)\\
		&=|A|\varphi^3|\nabla f|^2(\frac{|x|^2}{\varphi}-1)+|A|\varphi^3\nabla ^2f (\nabla f,x)-\varphi^4\sum_i \frac{A_i}{x_i}\nabla \left(\frac{\partial f}{\partial x_i}\right)\cdot \nabla f.
	\end{align*}
	It can be proved that in general it holds, even for arbitrary smooth functions $W_\theta$ and $f$, that (see \cite{BakryGentilLedoux})
	\begin{equation*}
		\nabla ^2_\theta W_\theta(\nabla_\theta f, \nabla_\theta f)=\Gamma^\theta(\Gamma^\theta(W_\theta,f),f)-\Gamma^\theta(\frac{\Gamma^\theta(f)}{2},W_\theta).
	\end{equation*}
	In our case in particular we have
	\begin{align}\label{eq4.3}
		\begin{split}
			\nabla ^2_\theta W_\theta(\nabla f, \nabla f)&=|A|\varphi^2|\nabla f|^2+|A|\varphi^2(x\cdot \nabla f)^2 -2\varphi^3\left(\sum_i \frac{A_i}{x_i}\frac{\partial f}{\partial x_i}\right)x\cdot \nabla f+\varphi^4\sum_i \frac{A_i}{x_i^2}\left(\frac{\partial f}{\partial x_i}\right)^2\\
			&=|A|\varphi^2|\nabla f|^2+\frac{1}{|A|}\Gamma^\theta(W_\theta,f)^2-\frac{\varphi^4}{|A|}\left(\sum_i\frac{A_i}{x_i}\frac{\partial f}{\partial x_i}\right)^2+\varphi^4\sum_i \frac{A_i}{x_i^2}\left(\frac{\partial f}{\partial x_i}\right)^2\\
			&=|A|\varphi^2|\nabla f|^2+\frac{1}{|A|}\Gamma^\theta(W_\theta,f)^2-\varphi^4 |A|\left(\sum_i\frac{A_i}{|A|}\frac{1}{x_i}\frac{\partial f}{\partial x_i}\right)^2+\varphi^4\sum_i \frac{A_i}{x_i^2}\left(\frac{\partial f}{\partial x_i}\right)^2\\
			&\geq |A|\varphi^2|\nabla f|^2+\frac{1}{|A|}\Gamma^\theta(W_\theta,f)^2=|A|\Gamma^\theta(f)+\frac{1}{|A|}\Gamma^\theta(W_\theta,f)^2
		\end{split}
	\end{align}
	where the last inequality is a consequence of the convexity of the map $x\mapsto x^2$ since $|A|=\sum_i A_i$. To conclude, recalling that 
	\[
	\text{Ric}_\theta(\nabla _\theta f,\nabla _\theta f)=(d-2)\Gamma^\theta (f),
	\]
	it is enough to use Bochner's formula and Cauchy-Schwartz's inequality to have
	\begin{align*}
		&\Gamma_2^\theta(f)-(D-2)\Gamma^\theta(f)-\frac{1}{D-1}(L_\theta f)^2\\
		&=||\nabla_\theta^2f||^2+\text{Ric}_\theta(\nabla_\theta f,\nabla_\theta f)+\nabla ^2_\theta W_\theta(\nabla_\theta f, \nabla_\theta f)-(D-2)\Gamma^\theta(f)-\frac{1}{D-1}(L_\theta f)^2\\
		&=||\nabla_\theta^2f||^2+\nabla ^2_\theta W_\theta(\nabla_\theta f, \nabla_\theta f)-|A|\Gamma^\theta(f)-\frac{1}{D-1}(L_\theta f)^2\\
		&\geq\frac{1}{d-1}(\Delta_\theta f)^2 +\nabla ^2_\theta W_\theta(\nabla_\theta f, \nabla_\theta f)-|A|\Gamma^\theta(f)-\frac{1}{D-1}(L_\theta f)^2\\
		&\overset{\eqref{eq4.3}}{\geq}\frac{1}{d-1}(\Delta_\theta f)^2 +\frac{1}{|A|}\Gamma^\theta(W_\theta,f)^2-\frac{1}{D-1}\big((\Delta_\theta f)^2+\Gamma^\theta(W_\theta,f)^2-2\Delta_\theta f \Gamma^\theta(W_\theta,f)\big)\\
		&=\frac{|A|}{(d-1)(D-1)}(\Delta_\theta f)^2-\frac{2}{D-1} \Delta_\theta f \Gamma^\theta(W_\theta,f)+\frac{d-1}{(D-1)|A|}\Gamma^\theta(W_\theta,f)^2\\
		&=\frac{1}{D-1}\left(\sqrt{\frac{|A|}{d-1}}\Delta_\theta f-\sqrt{\frac{d-1}{|A|}}\Gamma^\theta(W_\theta,f)\right)^2\geq 0,
	\end{align*}
	the thesis.
	\section{Acknowledgements}
	\lowcotwo\ This is a low-co2 research paper: \lowcotwourl[v1]. This research
	was developed, written, submitted and presented without the use of air travel.\\
	I would like to thank my advisor Louis Dupaigne for the precious discussions we had and all the advice he gave me during this year. Without him this work would not have been possible. I also would like to thank Nikita Simonov for the many references that he recommended and his fruitful suggestions. At last, I would like to thank Simon Zugmeyer and Ivan Gentil for having allowed me to read their paper in advance.\\ 
	This project has received funding from the European Union's Horizon Europe research and innovation programme under the Marie Sklodowska-Curie grant agreement No 101126554.
	\section{Disclaimer}
	Co-Funded by the European Union. Views and opinions expressed are however those of the author only and do not necessarily reflect those of the European Union. Neither the European Union nor the granting authority can be held responsible for them.
	
\end{document}